\documentclass[final]{econsocart}
\RequirePackage[colorlinks,citecolor=blue,linkcolor=blue,urlcolor=blue,pagebackref]{hyperref}

\startlocaldefs

\theoremstyle{plain}

\newtheorem{proposition}{Proposition}

\theoremstyle{remark}
\newtheorem{definition}{Definition}

\newtheorem{remark}{Remark}

\usepackage{bm}

\usepackage{enumitem}
\usepackage{mathtools}
\usepackage{graphicx}
\usepackage{amssymb}
\usepackage{dsfont}
\usepackage{epstopdf}
\usepackage{pgfplots}
\usepackage{comment}
\usepackage{paralist}

\providecommand{\reals}{\mathbb{R}}
\providecommand{\NN}{\mathbb{N}}

\providecommand{\eps}{\varepsilon}
\providecommand{\Prob}{\mathcal{P}} 
\providecommand{\widebar}[1]{\overline{#1}}
\newcommand{\prob}{\mathrm{P}}
\providecommand{\ind}{\mathds{1}}

\providecommand{\Fb}{\mathbf{F}}
\providecommand{\Fbpm}{\Fb_{\pm}}
\providecommand{\Qb}{\mathbf{Q}}

\providecommand{\grx}{\mathbf{x}}
\providecommand{\gry}{\mathbf{y}}
\providecommand{\gru}{\mathbf{u}}

\providecommand{\n}{^{(n)}}

\newcommand{\diff}{\mathrm{d}}
\newcommand{\Xb}{\mathbf{X}}
\newcommand{\Yb}{\mathbf{Y}}
\newcommand{\xb}{\mathbf{x}}

\newcommand{\ub}{\mathbf{u}}
\newcommand{\Mb}{\mathbf{M}}
\newcommand{\sbold}{\mathbf{s}}
\newcommand{\Sb}{\mathbf{S}}

\newcommand{\E}{{\mathbb E}}
\newcommand{\Ell}{{\mathcal L}}
\newcommand{\Kak}{{\mathcal K}}

\numberwithin{equation}{section}

\providecommand{\ball}{\mathbb{S}}
\providecommand{\sphere}{\mathcal{S}}
\providecommand{\expec}{\mathbb{E}}

\providecommand{\normal}{\mathcal{N}}
\providecommand{\supp}{\operatorname{supp}}

\providecommand{\cov}{\operatorname{Cov}}

\usepackage{hyperref}

\usepackage{color}
\definecolor{bulgarianrose}{rgb}{0.28, 0.02, 0.03}
\definecolor{darkraspberry}{rgb}{0.53, 0.15, 0.34}
\definecolor{britishracinggreen}{rgb}{0.0, 0.26, 0.15}
\definecolor{burntumber}{rgb}{0.54, 0.2, 0.14}
\definecolor{royalblue}{RGB}{0,78,156}

\endlocaldefs

\begin{document}

\begin{frontmatter}

\title{Center-Outward Multiple-Output \\Lorenz Curves and Gini Indices \\ a measure transportation approach}
\runtitle{Center-outward multiple-output Lorenz curves and Gini indices}

\begin{aug}
%
%
%
\author[id=au1,addressref={add1}]{\fnms{Marc}~\snm{Hallin}\ead[label=e1]{mhallin@ulb.ac.be}}
\author[id=au2,addressref={add2}]{\fnms{Gilles}~\snm{Mordant}\ead[label=e2]{mordantgilles@gmail.com}}
\address[id=add1]{%
\orgdiv{ECARES and D{\'{e}}partement de Math\'{e}matique},
\orgname{Universit\'{e} libre de Bruxelles }
}
\address[id=add2]{%
\orgdiv{Institut f{\"u}r Mathematische Stochastik}, %
\orgname{Universit\"{a}t G\"{o}ttingen} %
}
\end{aug}

\support{The second author gratefully acknowledges
financial support from the Deutsche Forschungsgemeinschaft through CRC1456.}
%

\begin{abstract}
Based on measure transportation ideas and the related concepts of  center-outward quantile functions, 
we propose multiple-output center-outward generali\-zations of the traditional univariate concepts of Lorenz and concentration functions, and  the related Gini and Kakwani coefficients. These new concepts have a natural  interpretation, either in terms of contributions of central (``middle-class'') regions to the expectation of some variable of interest, or in terms of the physical notions of work and energy, which sheds new light on the nature of economic and social inequalities. 
Importantly, the proposed concepts pave the way to  statistically sound definitions, based on multiple variables, of quantiles and quantile regions, and the concept of  ``middle class,''   of  high relevance  
 in various socio-economic   contexts. 
\end{abstract}

\begin{keyword}
\kwd{Center-outward quantiles}
\kwd{Measure transportation}
\kwd{Concentration indices}
\kwd{Inequality measurement}
\kwd{Definition of ``middle class''}
\end{keyword}

\end{frontmatter}

\section{Introduction}
\subsection{Lorenz curves and the measurement of inequality}
Measuring inequalities is a major issue in some areas of economics and the social sciences, and a vast literature has devloped on the subject, which even led to the publication of dedicated journals such as \textit{The Journal of Economic Inequality}. The most popular statistical tools in that context are the  various indices of concentration or inequality related with the so-called {\it Lorenz} and {\it concentration functions}  \citep{Lorenz1905,kakwani1977applications}, themselves in close relation  to  integrated quantile functionals. 

The concept of quantile is clear and well understood in dimension one, and 
  Lorenz functions originally were developed for  univariate quantities (single-output case) only. 
  Due to the lack of a canonical ordering in real spaces of  dimension two and higher,  however,   multivariate  extensions  (multiple-output case) of the concept are more problematic. The few existing attempts to define Lorenz functions in a multiple-output  context came much later and are, in general, not associated with any sound concept of multivariate quantile; their interpretation, therefore, is  not entirely satisfactory, and none of  them has met a general consensus among practitioners. 
  
  Based on measure transportation ideas, new definitions of multivariate distribution and quantile functions have been proposed recently \citep{chernozhukov2017monge, del2018center} which, contrary to earlier ones, are enjoying the fundamental properties one is expecting from a quantile function; see \citet{H2021} for a survey. In this measure transportation context, convex potentials naturally extend the notion of  integrated quantile functions. 
 Building on these concepts of multivariate quantiles and convex potentials, this paper is proposing  new  multivariate extensions of the classical  concepts of Lorenz and concentration functions and the related indices---the best-known of which is the Gini coefficient \citep{gini1912variabilita}. 
 
 \subsection{Organisation of the paper}
   The rest of the paper is organized as follows. 
We briefly start (Section~\ref{histsec}) by reviewing the history of Lorenz curves and inequality measurement. We then propose (Section~\ref{sec: OTQf}) an overview of center-outward quantile functions. 
 In Section~\ref{sec: OTLK}, we   introduce,  in dimension one first, then in arbitrary dimension $d$,  the concept of {\it center-outward} Lorenz function. Starting with dimension two, indeed, the space is no longer oriented from $-\infty$ to $\infty$, but each direction yields two points at infinity, none of which qualifies as $-\infty$ nor $\infty$. This leads to center-outward rather than ``left-to-right'' orderings. The resulting multivariate indices (Gini, Kakwani) are introduced in Section~\ref{SecGini}.  
 Section~\ref{sec: Est} proposes estimators of  the newly defined population quantities  and establishes their consistency properties. We conclude in Section~\ref{sec: App} with an application to a dataset relating income, life expectancy, and education. This application nicely illustrates how the  proposed concepts yield new quantitative descriptions of multivariate inequalities and concentrations.  
 
\section{A brief history of Lorenz curves and Gini coefficients}
\label{histsec} 

We start with  a brief review of some classical definitions and concepts   arising in  connection with  Lorenz curves  and the associated indices.  We then propose a general definition that encompasses most previous ones. Due to space considerations, this survey unavoidably is  incomplete; we refer to \citet{arnold2018majorization} for a book-long exposition. 
\subsection{Traditional single-output concepts}

The traditional definition of a Lorenz curve deals with a univariate real-valued  positive absolutely continuous random variable~$X$ with Lebesgue density  
 $f_X(x)>0$ for all $x\geq 0$.
  Denoting by $F_X$ its distribution function, $F_X(0)=0$ and~$F_X$ is strictly increasing, so that   the quantile function~$Q_X\coloneqq~\!F_X^{-1}$ is well defined. Assume that the mean
\[
0<\mu_X\coloneqq \mathbb{E} X\coloneqq\int_{-\infty}^\infty xf_X(x)\diff x=  \int_0^1 Q_X(t)\diff t  
\]
is finite: the classical definition of a Lorenz curve then is as follows \citep{Lorenz1905,gastwirth1971general}. 
\begin{definition}\label{Lor1}{\rm 
 The {\em (relative) Lorenz curve of }$X$ is   the graph~$\left\{(u, \Ell_X(u)),\, u\in [0,1]
 \right\}$ of the {\em (relative) Lorenz function}  
\begin{align}
\label{eq: altDefLC}
u\mapsto \Ell_X(u)&\coloneqq  \frac{\int_0^u Q_X(t)\diff t }{ \int_0^1 Q_X(t)\diff t} =\mu_X^{-1} \mathbb{E}\left[X \ind_{\{X\le Q_X(u)\}} \right]\nonumber \\ 
&= \mu_X^{-1}\left( \psi_X(u)-\psi_X(0)\right) , \quad 0\le u\le1
\end{align}
where $\psi_X: (0,1)\to\mathbb R$ is a primitive of $Q_X$, that is,  $Q_X(u)=\dfrac{{\rm d}\psi_X }{{\rm d} u}(u)$ for~$0<~\!u<~\!1$, with~$\psi_X (1)-\psi_X (0)
=\mu_X$. }
\end{definition}

It immediately follows from the assumptions that $\Ell_X$  is continuous, strictly increasing, and a.e.\ differentiable, with $0=\Ell_X(0)<\Ell_X(1)=1$; $\psi_X$ is a.e.\ differen\-tiable and convex, with derivative $Q_X$. Without any  loss of generality, let us impo\-se~$\psi_X(0)=0$: then,
\begin{equation}\label{potdefLC} \mu_X = \psi_X(1)\quad\text{and}\quad 
\Ell_X(u)=  \psi_X(u)/\psi_X(1), \quad 0\le u\le1.
\end{equation}

The numerator $u\mapsto L_X(u)\coloneqq \int_0^u Q_X(t)\diff t $ in \eqref{eq: altDefLC} and \eqref{potdefLC}---call it the {\it absolute Lorenz function}---represents the contribution to the inte\-gral~$\int_0^1 Q_X(t)\diff t=\mu_X$ of the quantile region $(-\infty, u]$, which provides the intuitive interpretation of the concept. This absolute definition  still makes sense 
  if the requirement of a positive $X$ is dropped. 
 
 In the sequel, the positivity and finite-mean assumptions are implicitly made whenever  relative Lorenz functions are mentioned and implicitly relaxed when absolute Lorenz functions are involved.

Associated with the relative Lorenz function  are several {\it concentration indices},  the most popular of which  is the Gini coefficient   \citep{gini1912variabilita},  defined as twice  the area comprised between the (relative) Lorenz curve and  the main diagonal of the unit square, namely, 
\begin{equation}
\label{eq: Gini}
G_X\coloneqq  2\int_0^1 \left[u-\Ell_X(u)\right] {\rm d}u.
\end{equation}
Other indices are the {\em Pietra  index} (\cite{pietra1915delle}; see \cite{pietra2014}) 
\[ 
P_X\coloneqq  \sup_{0\le u \le 1} [u- \Ell_X(u)]
\]
and the  {\em Amato-Kakwani index}, defined  as the rescaled length \[
AK_X\coloneqq  \frac{l_X -\sqrt{2}}{2 -\sqrt{2}}
\]
of the relative  Lorenz curve, 
where (denoting by $\Ell^\prime_X$ the a.e.\ derivative of the Lorenz function~$u\mapsto \Ell_X(u)$)
\[
l_X \coloneqq  \int_0^1 \sqrt{1+ (\Ell_X^\prime(u))^2} \,\diff u .
\]
These three indices all take values in $[0,1]$. They can be interpreted as a measure of discrepancy between two  {distribution functions    over $[0,1]$:} the uniform, with probability density~$f(u)=1$ 
  and the one 
   with 
 density $f(u)=Q_X(u)/\mu_X$,  $u\in[0,1]$. 

Other characterizations of the Gini index are possible. For instance, denoting by  $X_1$ and~$X_2$ two independent copies of $X$, it can be shown that \citep[p.~384]{yitzhaki1991concentration} 
 \begin{align}
\label{eq: GiniAlt}
G_X = \frac{\mathbb{E} \lvert X_1 -X_2\rvert}{2\mu_X} 
= \frac{1}{2\mu_X} \mathbb{E} \left\lvert \operatorname{det} \begin{pmatrix} 1 & X_1 \\ 1 & X_2 \end{pmatrix} \right \rvert
= \frac{2}{\mu_X} \operatorname{Cov}(X, F(X)),
\end{align}
a definition that does not explicitly involve the quantile function $Q_X$.

Denoting by $g: x\mapsto g(x)$ a continuous and positive    real-valued function with finite   mean~$\mu_g>0$, the following extension of the Lorenz function was proposed by \citet{kakwani1977applications} and is often considered  in economics.
\begin{definition}\label{Lor2}{\rm 
The {\it (relative) concentration function of} $g(X)$ with respect to $X$  is  
\begin{equation}\label{concfctdef}
u\mapsto  {\Kak}^g_X(u) \coloneqq \frac{ \int_0^u g(Q_X(t)) \diff t}{ \int_0^1 g(Q_X(t)) \diff t}  
=\mu_g^{-1} \mathbb{E}\left[g(X) \ind_{\{X\le Q_X(u)\}} \right] \quad u\in [0,1].
\end{equation}
}
\end{definition}
For  positive   and stricly increasing $g$, $\Kak^g_X(u)$ coincides with $\Ell_{g(X)}(u)$; in particular,~$g(x)=x$ yields  $\Kak^g_X=\Ell_X$. But $g$ needs not be positive and monotone, nor have finite mean: the numerator $K^g_X(u)\coloneqq\int_0^u g(Q_X(t)) \diff t$ in \eqref{concfctdef}---call it the {\it absolute concentration function of} $g(X)$ then still makes sense.  

A case that  has  attracted particular  interest in economics is  $g(x) \coloneqq \expec [Y \vert X=x]$ where~$Y$ is some variable of interest: $\Kak^g_X$ then yields  
 \begin{align}\nonumber
 u\mapsto \Kak_{Y/X}(u) &\coloneqq \frac{ \int_0^u {\rm E}[Y\vert X=Q_X(t)] \diff t}{ \int_0^1 {\rm E}[Y\vert X=Q_X(t)] \diff t}=\mu_Y^{-1} \int_0^u {\rm E}[Y\vert X=Q_X(t)]\diff t \\
 &= \mu_Y^{-1} {\rm E}\left[Y\ind_{\{X\le Q_X(u)\}} \right] =\mu_Y^{-1}\left(\psi_{Y/X}(u) -\psi_{Y/X}(0)\right)\quad u\in [0,1]
\label{KXYdef} \end{align}
where $\psi_{Y/X}$ denotes an arbitrary primitive of $u\mapsto {\rm E}[Y\vert X=Q_X(u)]$ (the value of which we can  assume to be zero at $u=0$). 

For any $g$, there exist  infinitely many variables $Y$ such that  $g(x) = \expec [Y \vert X=x]$; rather than a {\it particular case} of \eqref {concfctdef}, \eqref{KXYdef} thus should be considered as an alternative but equivalent expression involving the joint distribution of a couple of random variables $(X,Y)$. We emphasize this by calling $\Kak_{Y/X}$ and its numerator   
\[
u\mapsto K_{Y/X}(u)\coloneqq \int_0^u {\rm E}[Y\vert X=Q_X(t)] \diff t
\]
 the {\it relative} and {\it absolute  Kakwani functions}, respectively,  {\it of $Y$ with respect to $X$}.

Whenever relative Kakwani functions are mentioned or the notation $\mu_Y$ is used, we tacitly assume that $0<\mu_Y<\infty$.

\subsection{Generalized Lorenz functions}

The value at $u$ of the relative Kakwani function of  $Y$ with respect to $X$ represents the contribution of the quantile region~$(-\infty , Q_X(u)]$ of $X$ to the mean $\mu_Y$ of $Y$, thus splitting the roles:   $Y$ is the variable of interest, but $X$ determines the quantile regions involved in the definition of~$\Kak_{Y/X}(u)$. 

The same role-splitting can be performed in the definition of Lorenz functions, yielding the {\it relative} and {\it absolute  Lorenz functions}   {\it of $Y$ with respect to $X$}. More precisely, 
 using obvious notation 
 $F_{X}$, $F_{Y}$, $Q_{X}$, $Q_{Y}$, $\psi_X$, and $\psi_Y$ for the couple~$(X,Y)$ of random variables, define the {\it relative Lorenz function of $Y$ with respect to $X$} as
\begin{align}\nonumber 
u\mapsto \Ell_{Y/X}(u) &\coloneqq \frac{\int_0^{F_Y\circ Q_X(u)}Q_Y(t)\diff t}{\int_0^1Q_Y(t)\diff t} =\mu_Y^{-1} {\rm E}\left[Y\ind_{\{Y\leq Q_X(u)\}}
\right] \\ 
&= \mu_Y^{-1}\left(\psi_Y(F_Y\circ Q_X(u))
 - \psi(0)\right) \quad u\in [0,1],
\label{LXYdef}\end{align}
reducing to $\Ell_{Y/X}(u) =\psi_Y(F_Y\circ Q_X(u))/\psi_Y(1)$ if we assume, without loss of generality, that $\psi_Y(0)=0$. The numerator $L_{Y/X}$ of \eqref{LXYdef} is the {\it absolute Lorenz function of $Y$ with respect to $X$}. To the best of our knowledge, such concept has never been considered in the literature.

Both $\Kak_{Y/X}$ and $\Ell_{Y/X}$ extend the traditional Lorenz function $\Ell_{X}$, to which they reduce for $Y=X$ and $Y\stackrel{d}{=}X$ (with $\stackrel{d}{=}$ standing for equality in distribution), respectively. While~$\Kak_{Y/X}(u)$ evaluates the contribution to $\mu_Y$ of  the region~$(-\infty, Q_X(u)]\times {\mathbb R}$ of the~$(x,y)$ space (i.e., the quantile region of order $u$ of $X$), $\Ell_{Y/X}(u)$  
is about the contribution of~${\mathbb R}\times(-\infty, Q_X(u)] $  (the quantile region of order $F_Y\circ Q_X(u)$ of $Y$).  An important difference between them  
  is that $\Kak_{Y/X}$ depends on the joint distribution~${\rm P}_{XY}$ of $X$ and $Y$, while~$\Ell_{Y/X}$ only involves their marginals ${\rm P}_X$ and~${\rm P}_{Y}$. As a consequence,  $\Kak_{Y/X}$, contrary to ~$\Ell_{Y/X}$, cannot be expressed in terms of the marginal integrated quantile functions $\psi_X$ and $\psi_Y$. On the other hand, $\Ell_{Y/X}$ only makes sense if $X$ and $Y$ are comparable quantities, which is the case when $Y=X$, or when ${\rm P}_X$ and ${\rm P}_Y$ are, for instance, the income distributions within two given socio-economic groups or two countries.

\subsection{The multivariate case: classical extensions} Extending the above concepts to multivariate random variables $\Xb$ and $\Yb$ is, of course, highly desirable. Several attempts have been made to extend the definitions of the Lorenz function \eqref{eq: altDefLC} and the Gini index \eqref{concfctdef} to a multivariate setting. The difficulty, of course, lies in the fact that, in dimension~2 and higher, the real space is no longer canonically ordered: as a consequence, the very concept of quantile 
is problematic. Quoting  \citet[page~149]{arnold2018majorization}, 
\begin{quote}Extensions of the Lorenz curve concept to higher dimensions was long frustrated by the fact that the usual definitions of the Lorenz curve involved either order statistics or a quantile function of the corresponding distribution, neither of which has a simple multidimensional analog. 
\end{quote} An extension of $\Ell_{Y/X}$ and $\Kak_{Y/X}$ to couples of random vectors $({\bf X}, {\bf Y})$  looks even more challenging.

Therefore, rather than generalizing the classical quantile-based definition~\eqref{eq: altDefLC}, the   extensions found in the literature essentially consist in bypassing the role of quantiles in these definitions; as a result, the  intuitive interpretation of the univariate definitions sometimes gets lost, and the value of the resulting multivariate concept is disputable. We only briefly review some of them.

\citet[page 150]{arnold2018majorization}, for instance, propose,   for a positive bivariate random vector~$\Xb =(X_1,X_2)^\prime$  with density $f(x_1,x_2)$ and marginal  
 quantile functions 
 $Q_1$
  and~$Q_2$,   the bivariate relative Lorenz function 
\begin{equation}
\label{eq: ArnoldBiv} (u,v)\in[0,1]^2\!\mapsto\! \Ell_{X_1,X_2}(u,v)\!\coloneqq\!   \frac{1}{\expec (X_1X_2)} \int_0^{Q_1(u)}\!\! \int_0^{Q_2(v)}\hspace{-2mm} x_1x_2f(x_1,x_2)\, \diff x_1 \diff x_2
\end{equation}
and the corresponding {\it Lorenz surface} $\{(u, v,  \Ell_{X_1,X_2}(u,v)),\,  (u,v)\in[0,1]^2 \}$. This indeed constitutes a  technically correct bivariate extension of Definition~\ref{Lor1} and straightforwardly  extends to higher dimensions. However,    the focus 
 on the product~$X_1X_2$ is somewhat ad hoc;  the resulting order  has not been fully studied.

Starting from Equation~\eqref{eq: GiniAlt}, which characterizes the Gini index without resorting to quantiles, other authors have proposed  ingenious multivariate extensions of the Gini index skipping the definition of a Lorenz function. 
Denoting by~$\Xb, \Xb_1, \ldots, \Xb_d $ a~$(d+1)$-tuple of independent copies of a $d$-dimensional 
variable~$\Xb$ and  
 considering the $(d+1)\times (d+1)$ matrix 
\[
\Mb_{\Xb} \coloneqq  \begin{pmatrix}  1\ &1 & \ldots & 1 \\
\Xb\ & \Xb_1 & \ldots & \Xb_d
\end{pmatrix}^\prime, 
\]
\citet{koshevoy1997multivariate} define the {\it Volume Gini Index of} $\Xb$  as 
\begin{equation}
\label{eq: GiniMultKM}
V\!G_\Xb\coloneqq \frac{1}{(d+1)!}\mathbb{E}\left\lvert \operatorname{det} \Mb_{\Xb} \right\rvert.
\end{equation}
This idea of measuring concentration based on  volumes of  convex bodies actually goes back to  \citet{wilks1960multidimensional}, who proposed  a coefficient that coincides with~$V\!G_\Xb$ up to a scale  factor. 
\citet{oja1983descriptive} pointed out that this coefficient  is the expected vo\-lume of the simplex with  vertices   the random $(d+1)$-tuple $\Xb, \Xb_1, \ldots, \Xb_d $; 
Koshevoy and Mosler showed that the definition~\eqref{eq: GiniMultKM} corresponds to the volume of a {\it lift zonoid} \citep{mosler2002multivariate}. While reducing, in dimension one, to the traditional concept, the Volume Gini index, however, fails, in dimension two and higher, to enjoy the   intuitive interpretation of the traditional Gini concentration index.  

Concepts of multivariate inequality/concentration and their measurement still constitute an active area of research; see the recent works by \citet{Mosler2022, grothe21,andreoli2020unidimensional,sarabia2020lorenz,Medeiros2015}, and the references therein. 
The latest contribution by \citet{Mosler2022}  explores the possibility of aggregating the various attributes of welfare; a compelling approach that differs from but complements ours.  The idea of aggregating data is also considered in \citet{sanchez2020multivariate} to construct a notion of middle-class that relies on multiple attributes. 
By the time we were completing  this manuscript, the work by \citet{fan2022lorenz} came to our attention. The multivariate Lorenz curves concepts they are proposing also are based on measure transportation ideas, and are aiming at the same objectives as ours. However, their concepts of multivariate quantiles are not of the center-outward type and do not allow for a clear definition of quantile regions;    the interpretation of their Lorenz curves, therefore, is essentially different from ours. 
Moreover, they do not consider the notions of concentration and Lorenz curves of one variable with respect to another one.


\section{Center-outward distribution and quantile functions}
\label{sec: OTQf}
We briefly present here the concept of center-outward quantile underlying the definition our multiple-output Lorenz curves and Gini indices. 

\subsection{Quantiles and center-outward quantiles}Based on measure transportation ideas, new concepts of multivariate quantile  have been proposed recently \citep{chernozhukov2017monge,del2018center,H2021} under the name of {\it Monge--Kantorovich } and {\it center-outward quantiles}, respectively. Contrary to previous proposals,   these center-outward quantiles enjoy the essential properties one is expecting from quantiles. In particular,  
the  resulting quantile regions  are closed, connected, and nested, with   probability contents that do not  depend on the underlying distributions (which is not the case, e.g., with  depth regions).  A distinctive feature of the real space in dimension~$d>1$ is the nonexistence of a linear, left-to-right  ordering from $+\infty$ to $-\infty$: instead, each direction $\bf u \in {\mathcal S}_{d-1}$ defines  an arbitrarily remote point, suggesting   {\it center-outward} partial orderings rather than a unique complete linear ordering from $-\infty$ to $\infty$.  Familiar examples of center-outward orderings  are the  orderings associated with the geometry of elliptical distributions or  the various concepts of depth---see, e.g., \cite{ SerfZuo00, Serf02, Serf19},  \citet{ Zuo21}, or  \citet{DPK22}. 

The orderings associated with the measure-transportation-based concepts  of distribution and quantile functions developed in \citet{del2018center} are of a center-outward nature comparable to  measures of outlyingness related to these depth notions, and offer  appealing multivariate  center-outward generalizations of ranks, signs, and quantile functions. 
Center-outward quantile functions, in that context, are defined as gradients of a {\it convex potential}~$\psi$---a natural form of integrated quantile function. As we shall see, when substituted  for traditional integrated univariate quantiles in~\eqref{eq: altDefLC}, \eqref{KXYdef}, and~\eqref{LXYdef}, that potential $\psi$   yields  extensions of the univariate concepts of Lorenz and Kakwani functions and the related indices that preserve the intuitive interpretation of the univariate concepts.


\subsection{Univariate center-outward distribution and quantile functions} 

Before addressing the general case, let us first examine the consequences, on the familiar  univariate concepts, of a center-outward definition of quantile regions. Let the univariate random variable $X$~have Lebesgue-absolutely continuous distribution~${\rm P}_X$ with continuous distribution function $F_X$, and call~$F_{X\pm} \coloneqq 2F_X -1$ the {\it center-outward distribution function of $X$}. That center-outward distribution function~$F_{X\pm}$ is such that
\[
\vert F_{X\pm}(x)
\vert= \prob[x^\prime\leq X\leq x^{\prime\prime}]
\]
 with  $x^\prime$, $x^{\prime\prime}$ characterized by 
 \[
F_X(x^{\prime})=\min(F_X(x), 1-F_X(x))\quad\text{ and}\quad F_X(x^{\prime\prime})=\max(F_X(x), 1-F_X(x)).
\]
    While the traditional condition~$F_X(x)\leq\tau$ characterizes  {\it quantile regions} with probability~$\tau$ as nested halflines of the form  $(-\infty, F_X^{-1}(\tau)]$, the condition~$\vert F_{X\pm} (x)\vert\leq\tau$ characterizes  {\it quantile regions} with probability~$\tau$ as nested central interquantile  intervals $[x', x'']$ of the form 
 \begin{equation}
 \label{ctauuni}
{\mathbb C}_X(\tau )\coloneqq [ F_{X\pm}^{-1} (-\tau),  F_{X\pm}^{-1} (\tau)]=[F_X^{-1}((1-\tau)/2), F_X^{-1}((1+\tau)/2)].
 \end{equation}
The intuitive interpretation of these {\it center-outward quantile regions} is clear: while the traditional quantile region of order $\tau$ contains the   proportion $\tau$ of the smallest  values of the random variable $X\sim{\rm P}_X$, ${\mathbb C}_X(\tau )$ contains the proportion $\tau$ of its ``most central'' values.\footnote{In dimension $d=1$, it is easy to see that  the notions of ``most central'' and ``Tukey-deepest''  coincide, so that these center-outward quantile regions also coincide with the Tukey depth regions---an equivalence that no longer holds in dimensions $d\geq 2$, though.}

 Let $\rm P$ be Lebesgue-absolutely continuous. A celebrated  theorem by \citet{McCann95} tells us  that there exists a unique  {\it potential} $\psi^*_\pm$  such that $\psi_\pm^*(0) = 0$ with gradient $\nabla \psi_\pm^*$    {\it pushing~$\rm P$ forward\,}\footnote{We adopt this convenient terminology and the corresponding notation  from measure transportation, where a mapping $G$ is {\it pushing} a distribution $\rm P _1$ {\it forward} to a distribution  $\rm P _2$ (nota\-tion:~$G\#{\rm P _1}=\rm P _2$) if $Z\sim \rm P _1$ entails~$G(Z)\sim \rm P _2$.} to the uniform ${\rm U}_1$  over the unit ball~${\mathbb S}_1=~\!(-1,1)$ of~$\mathbb R$---namely, such that~$\nabla \psi_\pm^*(Z)\sim{\rm U}_1$ for any~$Z\sim \rm P$, which we write  as $\nabla \psi_\pm^*\#\rm P=~\!\rm U _d$.  
 
 Since 
 $F_{X\pm}$ is monotone increasing (hence the derivative of a convex function) and such that 
  $F_{X\pm}(X)\sim {\rm U}_1$ (hence pushes the distribution $\rm P _X$ of~$X$ forward to~${\rm U}_1$), McCann's theorem   entails~$F_{X\pm}=\nabla \psi_{X\pm}^*$ Lebesgue-a.e. Assuming that~$X$ has an everywhere positive density, the inverse $Q_{X\pm}$ of $F_{X\pm}$ is well defined and is pushing~${\rm U}_1$ forward to $\rm P _X$: call it the {\it center-outward quantile function of~$X$}.     McCann's theorem again implies that, for some unique convex potential~$\psi_{X\pm}$ such that~$\psi_{X\pm}(0)=0$, the center-outward quantile function~$Q_{X\pm} (u)\eqqcolon \nabla\psi_{X\pm} (u) = {\rm d}\psi_{X\pm} (u)/{\rm d}u$;    elementary al\-gebra yields that $\psi_{X\pm} (u)$ equals $2\psi_X((u+1)/2)$ where $\psi_X$, defined in Section~\ref{histsec}, is a primitive of~$Q_X$. 
  Up to this point, $\psi_{X\pm}$ is defined on $(-1,1)$ only; this domain is classically extended to the real line by letting~$\psi_{X\pm}( 1)\coloneqq\lim_{u\uparrow   1}\psi_{X\pm}(u)$, $\psi_{X\pm}(-1)\coloneqq\lim_{u\downarrow   -1}\psi_{X\pm}(u)$, and~$\psi_{X\pm}(u)\coloneqq\infty$ for $\vert  u\vert >1$. 
   
\subsection{Multivariate center-outward   quantiles and transportation to the unit ball
} 
Since a one-to-one correspondence between the traditional distribution function~$F_X$ and the center-outward $F_{X\pm}$, hence between $Q_X$ and~$Q_{X\pm}$, exists, 
 the concepts $F_{X\pm}$ and~$Q_{X\pm}$ of center-outward distribution and quantile functions are of limited interest in dimension~$d=1$: their major advantage is that their definitions   and their interpretation, unlike those of~$F_X$ and~$Q_X$,   readily extend to dimensions $d\geq 2$, for which  no sound counterparts to $F_X$ and $Q_X$ are available. 

Let $\Xb$ be a $d$-dimensional random vector with distribution $\rm P_\Xb$ in the class $\mathcal P _d$ of Lebesgue-absolutely continuous distributions over $\mathbb R ^d$. In view of the aforementioned theorem by McCann, a center-outward distribution function $\Fb_{\Xb\pm}$ extending the univariate one~$F_{X\pm}$ can be defined as follows. Denote by $\mathbb S _d$ and $\mathcal S _{d-1}$ the open unit ball in $\mathbb R ^d$ and   the unit hypersphere centered at the origin, respectively, by~${\rm U}_d$  the spherical uniform over the ball $\mathbb S _d$,  and by  ${\rm V}_{d-1}$
the uniform over $\mathcal S _{d-1}$.  For $d=1$, the spherical  uniform $\rm U_d=\rm U_1$ reduces to the Lebesgue uniform over the unit ball~$\mathbb S_1=(-1,1)$. For $d\geq 2$, however,  
   $\rm U _d$  is the product of the uniform $\rm U_{[0,1]}$ over $[0,1]$ (for the distances to the origin) and the uniform~${\rm V}_{d-1}$ over  $\mathcal S_{d-1}$ (for the directions), which no longer coincides  with the Lebesgue uniform.

 \begin{definition}{\rm 
 \label{definition: dfqf}
  Call {\it center-outward distribution function of} $\Xb$ the a.e.\ unique gradient of convex function~$\nabla\psi_{\Xb\pm}^*=:\Fb_{\Xb\pm}$ pushing $\rm P_\Xb$ to~$\rm U_d$, that  is, such that~$\Fb_{\Xb\pm}\#\rm P_\Xb=\rm U_d$. }
 \end{definition}
This is the definition  proposed in \citet{del2018center}, 
where $\Fb_{\Xb\pm}$ is shown to enjoy all the properties expected from a distribution function. In particular, for a measure $\rm P_\Xb$ in  the class~{$\mathcal P _d^\pm\subset \mathcal P _d$} of all absolutely continuous distributions   with   density $f$   such that for all~$c>0$ there exists 
 $0\leq\lambda_c\leq\Lambda_c<\infty$ for which  
 $$\lambda_c\leq\inf_{\xb\in c\mathbb S _d} f(\xb)
\leq\sup_{\xb\in c\mathbb S _d}f(\xb)\leq~\!\Lambda_c,$$
 the center-outward distribution function~$\Fb_{\Xb\pm}$ is a homeomorphism between the punctured unit ball~$\mathbb S_d\setminus~\!\{{\boldsymbol 0}\}$ and~$\mathbb R ^d\setminus \Fb_\pm^{-1}({\boldsymbol 0})$ \citep{Fig18}; the condition $\rm P\in\mathcal P _d^\pm$ is partially relaxed in \citet{delB20} and \citet{del2018center}, where convex supports are allowed. Then, the center-outward distribution function $\Fb_{\Xb\pm}$ is continuously invertible except, possibly, at the origin: denote by $\Qb_{\Xb\pm}$ its inverse. This inverse  is the $\rm U _d$-a.s.\ unique gradient~$\nabla\psi_{\Xb\pm}$ of a convex potential function from $\mathbb S _d$ to $\mathbb R ^d$ pushing~$\rm U _d$ forward to $\rm P_\Xb$, and naturally qualifies as the center-outward quantile function of~$\Xb$ (see \citet{del2018center} for details). 

The function $\psi_{\Xb\pm}$, which so far is defined on $\mathbb S _d$ only, is extended to the closed ball~$\overline{\mathbb S}_d$ by  lower semi-continuity: $\psi_{\Xb\pm}(\mathbf{u}):=\liminf_{\mathbf{z}\to\mathbf{u},|\psi_{\Xb\pm} (\mathbf{z})|<1} \psi_{\Xb\pm} (\mathbf{z})$ for~$|\mathbf{u}|=~\!1$ and to $\mathbb R ^d$ by setting~$\psi_{\Xb\pm}(\mathbf{u}):=+\infty$ for $\mathbf{u}\notin \overline{\mathbb{S}}_d$  (see, e.g. (A.18) in~\cite{Fig18}). With this extension,  the potentials   $\psi_{\Xb\pm}$ and  $\psi_{\Xb\pm}^*$ are   Fenchel--Legendre conjugates, i.e., satisfy 
\[
\psi_{\Xb\pm}^*(\grx) \coloneqq  \sup_{\gru\in \ball_d}\left\{ \gru\cdot \grx - \psi_{\Xb\pm}(\gru) \right\},
\]
(throughout, we use the dot product notation $\gru\cdot \gry$     for the scalar product of $\gru$ and~$\grx$).

Finally,  the potentials $\psi_{\Xb\pm}$ and $\psi_{\Xb\pm}^*$ are uniquely determined if, without any loss of generality, we impose $\psi_{\Xb\pm}({\boldsymbol 0}) = 0 $.   \citet{Fig18} also shows that for spherically symmetric distributions, i.e., for distributions with densities  $f_\Xb(\xb)$ of the form $f_{\text{\rm rad}}(\Vert\xb -\boldsymbol\mu_\Xb\Vert)$ for some $\boldsymbol\mu_\Xb$,    the center-outward quantile function has the particular form 
\[
  \Qb_{\Xb\pm}(\ub) = \boldsymbol\mu_\Xb + q(\lVert \ub\rVert) \frac{\ub}{\lVert \ub\rVert},\quad \ub\in\ball_d
\]
where $q$ is the traditional univariate quantile function of $\Vert\Xb - \boldsymbol\mu\Vert$.  The quantile region of order $\tau$ of ${\rm U}_d$, thus, is the closed ball with radius $\tau$ centered at the origin.

Associated with $ \Fb_{\Xb\pm}$ and $ \Qb_{\Xb\pm}$ are the {\it quantile regions} ${\mathbb C}_\Xb(\tau)\coloneqq\Qb_{\Xb\pm} (\tau{\mathbb S}_d)$ and {\it contours} ${\mathcal C}_\Xb(\tau)\coloneqq\Qb_{\Xb\pm} (\tau{\mathcal S}_{d-1})$, $\tau\in(0,1)$;  the set  ${\mathcal C}_\Xb(0)\coloneqq\cap_{\tau\in(0,1)}{\mathbb C}_\Xb(\tau)$, which \citep{Fig18} is convex, compact, and has Lebesgue measure zero, can be considered as a {\it multivariate median}.

The intuition behind this definition of  the center-outward quantile regions  is the same as in the univariate case \eqref{ctauuni}: ${\mathbb C}_\Xb(\tau)$ contains the proportion $\tau$ of the ``most central''   values  of~$\Xb\sim {\rm P}_\Xb$. That intuition is illustrated, for $d=2$,  in Figure~\ref{Fig0}.

\begin{figure}[h!]
\includegraphics[trim=12mm 15mm 10mm 15mm, clip,width=6cm, height=6cm]{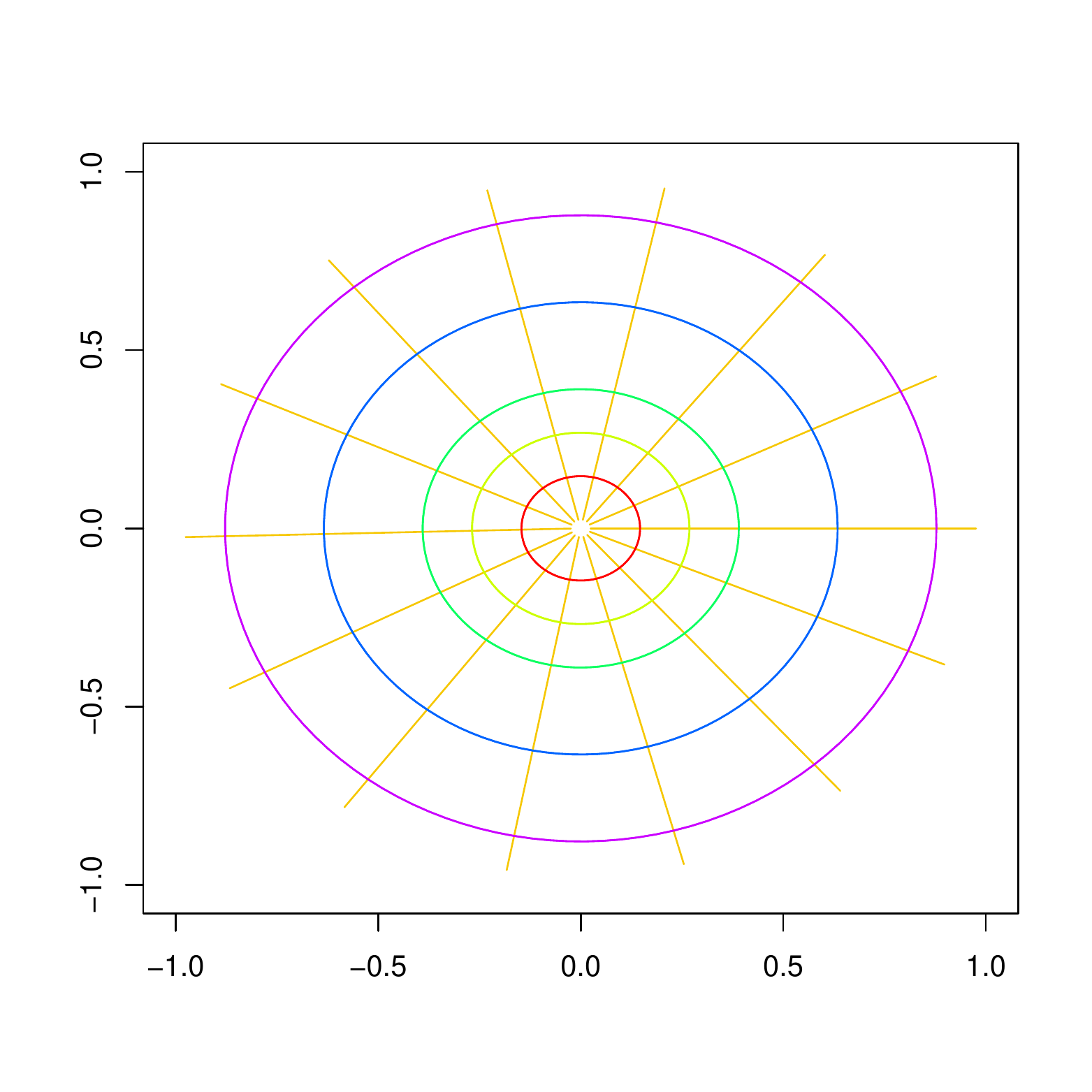} \ 
\includegraphics[trim=12mm 15mm 10mm 15mm, clip,width=6cm, height=6cm]{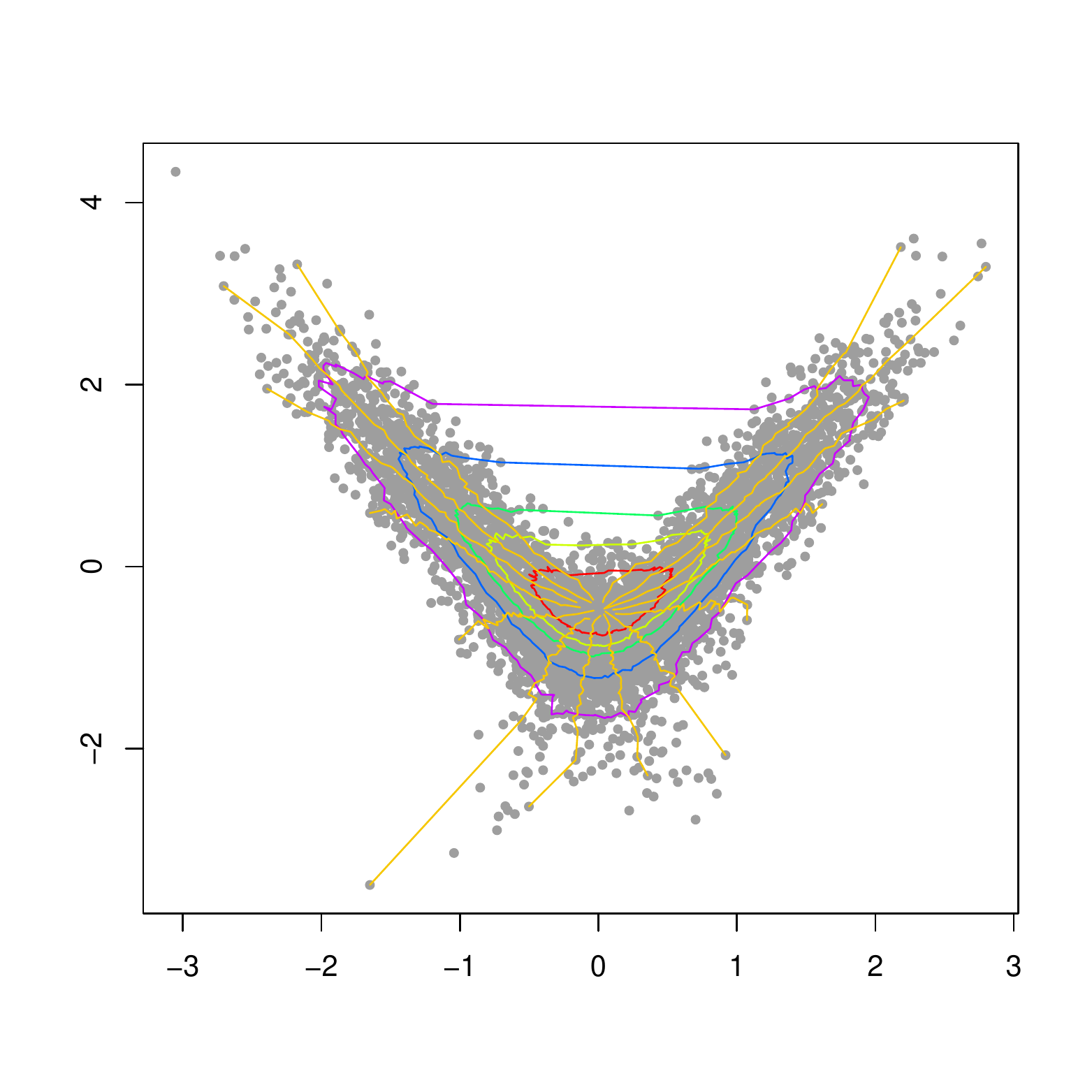} 
 \definecolor{color1}{HTML}{FF0000}
   \definecolor{color2}{HTML}{CCFF00}
     \definecolor{color3}{HTML}{00FF66}
     \definecolor{color4}{HTML}{0066FF}
       \definecolor{color5}{HTML}{CC00FF}
\fbox{ $\tau =$  \textcolor{color1}{\textemdash} 0.146 \quad
 \textcolor{color2}{\textemdash} 0.268\quad
 \textcolor{color3}{\textemdash} 0.39 \quad
 \textcolor{color4}{\textemdash}  0.634 \quad
  \textcolor{color5}{\textemdash} 0.878\quad
 }
\caption{Left panel: the two-dimensional unit ball, with its quantile regions of order $\tau = 0.146$, 0.268, 0.39, 0.634, and 0.878 (the balls~$\tau\bar{\mathbb{S}}_d$ with radius $\tau$ centered at the origin).  Right-hand panel: a numerical approximation, based on a sample of points generated from a banana-shaped mixture $\rm P$ of three normals, of the quantile regions~$Q_\pm \left(\tau\bar{\mathbb{S}}_d
\right)$ (same  $\tau$ values) of  $\rm P$.}
\label{Fig0}
\end{figure}

\subsection{{\it{\bf Vector quantiles}} and transportation to the unit cube}

The definition of center-outward quantiles in dimension $d\geq 2$ is based on a transport pushing ${\rm P}_{\Xb}$ forward to the spherical uniform ${\rm U}_d$ over the unit ball. Several authors (including  \citet{Carlier2016,chernozhukov2017monge, deb2019multivariate,ghosal2019multivariate}) rather privilege a transport to the Lebesgue uniform over the unit cube $[0,1]^d$. That choice produces, under the names of {\it vector ranks} and {\it vector quantile functions}, alternative $d$-variate distribution and quantile functions $\Fb_{\square}$ and $\Qb_{\,\square}\coloneqq \Fb_\square^{-1}$, say. 

As long as the objective is a characterization (in the empirical case, a consistent estimation) of ${\rm P}_X$ or (in the empirical case) the construction of distribution-free tests, that choice of a cubic uniform reference is perfectly fine. However, when it comes to defining quantile regions and contours, $\Qb_{\,\square}$ is running into the same conceptual difficulties as the inverse~$\Qb\coloneqq\Fb^{-1}$ of the  traditional multivariate distribution function 
$$\Fb: \ \xb=(x_1,\ldots,x_d)\mapsto \Fb(\xb )\coloneqq {\rm P}_\Xb\big[(-\infty,x_1]\times\ldots\times~\!(-\infty,x_d]
\big].$$
 Just as~$\Qb$,  $\Qb_{\,\square}$ is based on a $d$-tuple of marginal linear orderings rather than a unique center-outward one, and depends on a  $d$-tuple $(\tau_1,\ldots,\tau_d)$ of marginal orders rather than a unique $\tau$. This does not yield  (see, e.g., \citet{Genest01}) a satisfactory concept of quantile regions and contours. An inconvenient feature of unit cubes,  moreover,   is that there are many of them: besides
$[0, 1]^d$, constructed over the canonical coordinate system, all orthogonal transformations of
$[0, 1]^d$, yielding an infinite number of plausible quantile functions $\Qb_{\,\square}$,  are equally natural. All these~$\Qb_{\,\square}$'s carry the same information; as quantiles, however, their  interpretations may be extremely different. For instance, in dimension $d=2$, the same point $\xb$ can be of the form   $\xb=\Qb_{\,\square}^{(1)} (\tau_1\approx 1 ,\tau_2\approx 0)$ for some given unit square, of the form~$\xb=\Qb_{\,\square}^{(2)} (\tau_1\approx 0 ,\tau_2\approx 1)$ for some other one, and finally of the form~$\xb~=~\Qb_{\,\square}^{(3)} (\tau_1\approx 1/\sqrt{2} ,\tau_2\approx 1/\sqrt{2} )$ for a third one: in the first two cases, it is an extreme but in the third case it is not.    Finally, 
$\Qb_{\,\square}$, just as $\Qb$, is highly non-equivariant under orthogonal transformations of $\Xb$ while, thanks to the rotational invariance of ${\rm U}_d$ and the unit ball (see Proposition~2.2 in \cite{HHH}),   $\Qb_\pm$ is.

\cite{fan2022lorenz}  are building their approach to multiple-output Lorenz functions on such a~$\Qb_{\,\square}$,  which does not really allow for any interpretation based on  the contribution of the proportion $\tau$ of   ``most central,''    ``smallest,'' or ``less extreme''  values  of~$\Xb\sim {\rm P}_\Xb$.


\section{Center-outward Lorenz and Kakwani functions}
\label{sec: OTLK}

When generalizing the classical univariate definitions  \eqref{eq: altDefLC} and \eqref{KXYdef} of Lorenz and Kakwani  functions,  two distinct points of view can be adopted: 
\begin{compactenum}
\item[(a)] either   these classical functions are seen as   expectations of a family of  truncated random variables, e.g., of the form $X\ind_{\{X\in (-\infty,\, Q_X(u)]\}}$, or
\item[(b)]  as  integrals of  quantile functions, with the physical interpretation of  variations of potential or {\it work}.   
\end{compactenum}
Each of these two interpretations yields  center-outward versions which we  develop in Sections~\ref{incshsec}and~\ref{potfuncsec}, respectively.

\subsection{Center-outward  Lorenz, Kakwani,  and   \textbf{\textit{income share} functions}}\label{incshsec}

Emphasis, in the classical univariate definition \eqref {eq: altDefLC} of the Lorenz function, is on 
 quantile regions of the form~$(-\infty, Q(u)]$: the numerator, in \eqref{eq: altDefLC}, is the contribution  of the quantile 
 region~$(-\infty, Q(u)]$ to the mean $\mu_X\coloneqq \int_0^1Q(t) {\rm d}t  
$ (assumed to be finite). 
 In the center-outward approach,  emphasis is on  center-outward quantile regions of the form  (see~\eqref{ctauuni})
 \[
{\mathbb C}_\pm(\tau)\coloneqq  [ Q_\pm  (-\tau),  Q_\pm  (\tau)]=[Q((1-\tau)/2), Q((1+\tau)/2)] 
\]
 and their contribution to  $\boldsymbol\mu_\Xb$. This change of paradigm suggests the following definitions.

Let the random vectors $\Xb$ and $\Yb$  take values in ${\mathbb R}^{d_\Xb}$ and $ {\mathbb R}^{d_\Yb}$, respectively.  The notation~${\rm P}_{\Xb}$, $\rm P_\Yb$, and, whenever $\Xb$ and $\Yb$ are defined on the same probability space, ${\rm P}_{\Xb\Yb}$ (all assumed to be Lebesgue-absolutely continuous) is used in an obvious way, as well\linebreak as~$\Fb_{\Xb\pm}$, $\Qb_{\Xb\pm}$, $\Fb_{\Yb\pm}$, and $\Qb_{\Yb\pm}$. 

\begin{definition}
\label{c-oLor2} (Center-outward Lorenz and Kakwani functions)
\begin{compactenum}
\item[(i)]{\rm 
Call {\em (absolute) center-outward Lorenz function of }$\Xb$ the map-\linebreak ping~$u\mapsto L_{\Xb\pm}(u)$, where 
\begin{align}
\label{eq: coDefLC}
 L_{\Xb\pm}(u)&\coloneqq  \mathbb{E}\left[\Xb \ind_{\{ \Xb \in {\mathbb C}_{\Xb\pm}(u)\}} \right] , \quad 0\le u\le1;
\end{align}
}
\item[(ii)]  assuming   $d_\Xb = d=d_\Yb$,  {\em call (absolute) center-outward Lorenz function of~$\Yb$ with respect to $\Xb$} the mapping~$u\mapsto L_{\Yb/\Xb\pm}(u)$, where 
\begin{align}
\label{eq: coDefLC'}
	  L_{\Yb/\Xb\pm}(u)&\coloneqq   \mathbb{E}\left[\Yb \ind_{\{\Yb\in {\mathbb C}_{\Xb\pm}(u)\}} \right] , \quad 0\le u\le1;
\end{align}
\item[(iii)] call {\em (absolute) center-outward Kakwani function of~$\Yb$ with respect to $\Xb$}  the mapping~$u\mapsto K_{\Yb/\Xb\pm}(u)$, where
\begin{equation}\label{eq: Kakw}
  K_{\Yb/\Xb\pm}(u)\coloneqq  \expec \left[\Yb \ind_{\{\Xb\in {\mathbb C}_{\Xb\pm}(u)\}} \right] , \quad 0\le u\le1.
\end{equation}
\end{compactenum}
\end{definition}
Note that (iii), unlike~(ii), involves a joint distribution ${\rm P}_{\Xb\Yb}$  for $\Xb$ and $\Yb$. 

Instead of trimmed expectations, as in \eqref{eq: coDefLC}--\eqref{eq: Kakw},  one also may like to consider conditional expectations and replace, e.g., $\mathbb{E}\left[\Xb \ind_{\{ \Xb \in {\mathbb C}_{\Xb\pm}(u)\}} \right]$ with 
$$\mathbb{E}\left[\Xb \big\vert\,  \Xb \in {\mathbb C}_{\Xb\pm}(u) \right] = \mathbb{E}\left[\Xb \ind_{\{ \Xb \in {\mathbb C}_{\Xb\pm}(u)\}} \right]/u ,$$
 yielding {\it conditional center-outward Lorenz and Kakwani functions} which sometimes are easier to interpret.\vspace{-2mm}

\begin{definition}
\label{c-oLorCo} (Conditional center-outward Lorenz and Kakwani functions) 

Call {\it (absolute) conditional center-outward Lorenz function of} $\Xb$, {\em (absolute) conditional center-outward Lorenz function of~$\Yb$ with respect to} $\Xb$, and {\em (absolute) conditional center-outward Kakwani function of~$\Yb$ with respect to} $\Xb$ the functions 
\begin{equation}\label{defcond} 
 L^{\text{\tiny cond}}_{\Xb\pm}(u)\coloneqq L_{\Xb\pm}(u)/u ,\ 
  L^{\text{\tiny cond}}_{\Yb/\Xb\pm}(u)\coloneqq   L_{\Yb/\Xb\pm}(u)/u ,
  \text{ and } 
   K^{\text{\tiny cond}}_{\Yb/\Xb\pm}(u)\coloneqq K_{\Yb/\Xb\pm}(u)/u,\vspace{-3mm}
\end{equation}
 respectively.
\end{definition}
Conditional Lorenz and Kakwani functions convey the same information as their unconditional counterparts. In the absence of concentration, these conditional functions take value one for all $u$.

Just as $\Xb$ and $\Yb$ themselves, these center-outward functions are vector-valued ($L_{\Xb\pm}$ and~$L^{\text{\tiny cond}}_{\Xb\pm}$ are $d_\Xb$-dimensional, $L_{\Yb/\Xb\pm}$, $L^{\text{\tiny cond}}_{\Yb/\Xb\pm}$,  $K_{\Yb/\Xb\pm}$, and $K^{\text{\tiny cond}}_{\Yb/\Xb\pm}$   are~$d_\Yb$-dimensional).  Relative versions~$\mathcal  L_{\Xb\pm}$, $\mathcal  L^{\text{\tiny cond}}_{\Xb\pm}$, $\mathcal  L_{\Yb/\Xb\pm}$, $\mathcal  L^{\text{\tiny cond}}_{\Yb/\Xb\pm}$, $\mathcal  K_{\Yb/\Xb\pm}$, and~$\mathcal  K^{\text{\tiny cond}}_{\Yb/\Xb\pm}$  can be obtained by dividing each component of the corresponding  Lorenz and Kakwani absolute functions by its value at~$u=~\!1$ (provided that the latter are finite). These relative functions then define curves running  from~$(0,\ldots,0)$ to~$(1,\ldots,1)$ for $u=1$ which, if  the components of $\Xb$ (respectively, the components of~$\Yb$) all are non-negative, lie in the unit cube. These curves offer much information on the contributions of the central  quantile  regions to the mean of~$\Xb$ or the mean of~$\Yb$, with a large variety of possible behaviors. While a curve that coincides with the diagonal (for instance,~${\mathcal L}_{\Xb\pm}(u)= (u,\ldots,u)^\prime$ for $0\leq u\leq 1$) indicates the absence of concentration\footnote{The center-outward quantile function $\Qb_{\Xb \pm}$ of a degenerate random variable $\Xb$ taking   value~${\boldsymbol\mu}_{\bf X}$ with probability one ($\Xb\sim \delta_{{\boldsymbol\mu}_{\bf X}}$, the Dirac distribution at ${\boldsymbol\mu}_{\bf X}$) can be defined as  mapping any $ \mathbf{s} \in \mathbb{S}_d$ to $ {\boldsymbol\mu}_{\bf X}$. Indeed, there is only one way to push $\rm{U}_d$ forward to $\delta_{{\boldsymbol\mu}_{\bf X}}$. Then,  $L_{\Xb\pm}(u)= \int_0^u \int_{\mathcal{S}_{d-1}} \Qb_{\Xb \pm}( t\mathbf{s})\; \diff t \diff {\rm V}_{d-1}(\mathbf{s}) $  is well defined and equal to $u \boldsymbol{\mu}_\Xb$ (hence $L^{\text{\tiny cond}}_{\Xb\pm}(u)=\boldsymbol{\mu}_\Xb$). }, deviations the diagonal values  may yield very diverse interpretations:  for instance, an $i$th component of $  {\mathcal L}_{\Xb\pm}(u)- (u,\ldots,u)^\prime$   positive for~$0\leq u<u_0$, then negative for $u_0 <u\leq 1$ indicates a high contribution of $\Xb$'s  central quantile regions  (the ``middle class'') to the mean of $\Xb$'s $i$th components, with the more extremal quantiles contributing less. Mutatis mutandis, similar conclusions can be made for~${\mathcal L}_{\Yb/\Xb\pm}$ and ${\mathcal K}_{\Yb/\Xb\pm}$, allowing for a very detailed analysis of concentrations. Importantly, this analysis incorporates the dependencies between the components of the random vector, enabling a genuine notion of joint multivariate centrality.

 \begin{remark}
 The word ``concentration'' is not to be understood in the sense of concentration of probability measures. To comply with the traditional terminology of Lorenz curves where concentration refers to the concentration of income in the hands of a few rich people, we also describe the situation in which all individual incomes are the  same \textemdash i.e.,  $\Xb={\boldsymbol\mu}_{\bf X}$  almost surely\textemdash as the absence of concentration. 
\end{remark}

Center-outward versions of relative    Lorenz   functions actually have been  considered before, under a slightly different form, for  univariate non-negative variables.  For instance, \citet{davidson2018statistical}, in a study of  household income in  Canada, proposes an analysis of the {\it income share function}  ($\text{\it IS}_X$) of the ``middle class'' based on the {\it income share function}
\begin{equation}
\label{eq: ISClass}
\text{\it IS}_X(a,b)\coloneqq  \mu_X^{-1}\int_{aX_{(1/2)}}^{bX_{(1/2)}} x \diff F_X(x),
\end{equation}
where $X$ denotes the (nonnegative) revenue of a  Canadian household chosen at random,~$X_{(1/2)}\coloneqq Q_X(1/2)$ is the median of $X$, and $0\leq a < b$. 
For the sake of consistency, the version we are providing here is the relative  one which, of course, requires $0< \mu_X<\infty$.  
The relative center-outward Lorenz function 
 for a univariate random variable $X$
\begin{align}\label{ISpmdef}
\Ell_{X\pm}(u) &=  
  \frac{1}{2\mu_X}
   \left[
    \int_{[0,u]}Q_{X\pm}(t)\,\diff {\rm U}_{[0,1]} (t) +  \int_{[0,u]}Q_{X\pm}(-t)\,\diff {\rm U}_{[0,1]} (t)
\right] \\ 
&= \mu_X^{-1} \int_{\left[\frac{1-u}{2},\frac{1+u}{2}\right]} Q_X(t) \,\diff {\rm U}_{[0,1]} (t)
=\text{\it IS}_X\left(\frac{Q_X(\frac{1-u}{2})}{X_{(1/2)}},\frac{Q_X(\frac{1+u}{2})}{X_{(1/2)}}\right), \  \  u\in[0, 1)\nonumber
\end{align}
offers a natural quantile-based slight modification of Davidson's concept~\eqref{eq: ISClass}; the equal-tails choice of the  integration domain, moreover, provides a   sound justification of the  terminology  {\it middle-class}.

We finally stress that the concepts proposed in Definition~\ref{c-oLor2}  fill a gap by providing a new tool that has been long awaited. 
Indeed, \citet{davidson2018statistical} concludes that 
 \begin{quote}
 An ideal definition [of middle class]  would have to be based on all sorts of socio-economic characteristics of individuals and households ...
 \end{quote}  
Due to the lack of a multivariate version of  the central intervals (of a form  similar to~$[aX_{1/2}, bX_{1/2}]$) defining the ``middle class,''  he regretfully   concludes that, desirable as it is,  defining a generalized income share function  based on such a multivariate characterization of the middle class is {\it ``well beyond the scope of this work.'' } 
The same concern is expressed by \citet{atkinson2013identification} in their study of the evolution of the concept of Middle Class.
   The definitions, for univariate $Y$ and multivariate $\Xb$ (the socio-economic variables characterizing the ``middle class''), of the center-outward Kakwani functions $K_{Y/\Xb\pm}$ and $K^{\text{\tiny cond}}_{Y/\Xb\pm}$ of~$Y$ with respect to $\Xb$ and their relative counterparts    are providing  the perfect solution to that need. 
 Figure~\ref{fig: LorComp} provides,  for three classical distributions,  the classical Lorenz functions along with the newly proposed center-outward ones.
 
 \begin{figure}[h!]
\centering
\begin{tikzpicture}[]
\begin{axis}[
width=6cm,
 height=6cm,
    axis lines = left,
    xlabel = {$u$ },
    ylabel = {$\mathcal{L}_{X\pm}^{\text{cond}}(u)$},
    legend style={at={(0.65,1)},anchor=north east}
]
\addplot [
    domain=0:1, 
    samples=300, 
    color=blue,
]
{  ((x*(2+ln(4))-(-1+x)*ln(1-x)-(1+x)*ln(1+x))/2)/x };
\addplot [
    domain=0:1, 
    samples=300, 
    color=red,
]
{  1 };
\addplot [
    domain=0:1, 
    samples=300, 
    color=black,
]
{ (sqrt(2)*(sqrt(1+x) -sqrt(1-x) )/2)/x };
\end{axis}
\end{tikzpicture}
\begin{tikzpicture}[]
\begin{axis}[
width=6cm,
 height=6cm,
    axis lines = left,
    xlabel = {$u$ },
    ylabel = {$\mathcal{L}_{X\pm}(u)$},
    legend style={at={(0.65,1)},anchor=north east}
]
\addplot [
    domain=0:1, 
    samples=300, 
    color=blue,
]
{  (x*(2+ln(4))-(-1+x)*ln(1-x)-(1+x)*ln(1+x))/2 };
\addplot [
    domain=0:1, 
    samples=300, 
    color=red,
]
{  x };
\addplot [
    domain=0:1, 
    samples=300, 
    color=black,
]
{ sqrt(2)*(sqrt(1+x) -sqrt(1-x) )/2 };
\end{axis}
\end{tikzpicture}
\begin{tikzpicture}[]
\begin{axis}[
width=6cm,
 height=6cm,
    axis lines = left,
    xlabel = {$u$},
    ylabel = {$\mathcal{L}_X(u)$}
]
\addplot [
    domain=0:1, 
    samples=300, 
    color=blue,
]
{  -(x-1)*ln(1-x)+x };
\addplot [
    domain=0:1, 
    samples=300, 
    color=red,
]
{  x^2 };
\addplot [
    domain=0:1, 
    samples=300, 
    color=black,
]
{ 1- sqrt(1-x)  };
\end{axis}
\end{tikzpicture} 
\begin{tikzpicture}[]
\begin{axis}[
width=6cm,
 height=6cm,
    axis lines = left,
    xlabel = {$u$ },
    ylabel = {$\Lambda_{X\pm}(u)/\Lambda_{X\pm}(1)$},
    legend style={at={(0.65,1)},anchor=north east}
]
\addplot [
    domain=0:1, 
    samples=300, 
    color=blue,
]
{  ( -(x-1)*ln((1-x)/2)/4 + (1+x)*ln((1+x)/2)/4 + ln(2)/2 ) *2/ln(2) };
\addplot [
    domain=0:1, 
    samples=300, 
    color=red,
]
{ x^2 };
\addplot [
    domain=0:1, 
    samples=300, 
    color=black,
]
{ (- ( sqrt(1+x)+ sqrt(1-x) )/(2*sqrt(2)) +1/sqrt(2))/ (1/sqrt(2)-1/2) };
\end{axis}
\end{tikzpicture}
\fbox{  \textcolor{black}{\textemdash} Pareto  {$\left(x_m=\frac{1}{4}, \alpha=2\right)$} \quad
 \textcolor{blue}{\textemdash} Exponential $(\lambda=2)$\quad
 \textcolor{red}{\textemdash} Uniform on [0,1] }
\caption{\small\slshape
 Top: center-outward relative conditional Lorenz functions $\mathcal{L}_{X\pm}^{\text{\rm cond}}$ and  center-outward  relative Lorenz functions $\mathcal{L}_{X\pm}$. Bottom left: classical Lorenz functions. Bottom right: center-outward  relative Lorenz potential functions $\Lambda_{X\pm}(u)/\Lambda_{X\pm}(1)$. The parameters are chosen such that each distribution has expected value  1/2.}   
\label{fig: LorComp} \vspace{-6mm}
\end{figure}
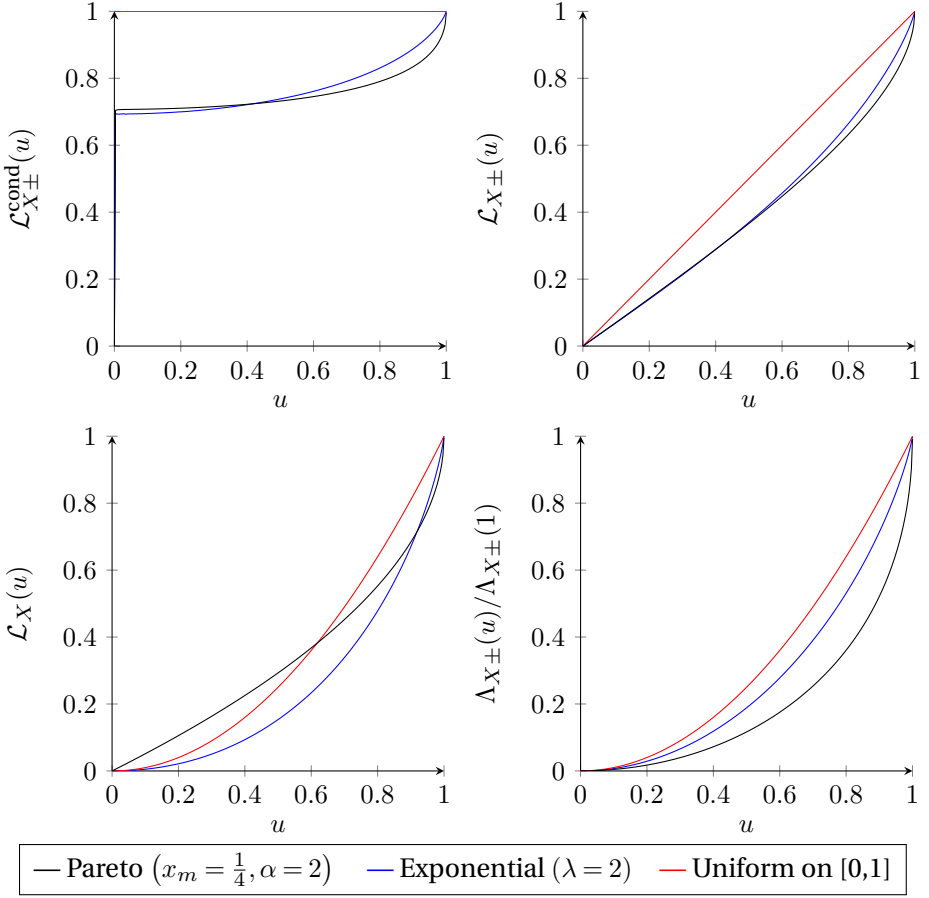

\subsection{Center-outward  Lorenz and Kakwani   \textbf{\textit{potential}} functions}
\label{potfuncsec}

An alternative point of view on \eqref{eq: altDefLC} consists of emphasizing the interpretation of the classical univariate Lorenz and Kakwani functions as variations of a potential  function $\psi$. 
 That potential describes how the mass of the (uniform) reference measure is  optimally (monotonically, in the sense of cyclical monotonicity) reorganized into the probability distribution~${\rm P}_\Xb $ of~$\Xb$. The potential variation $\psi(\ub_1)-\psi(\ub_2)$   between two points~$\ub_1$ and~$\ub_2$  can be understood as the work to be exerted on an object of unit mass that has to travel along any arbitrary smooth\footnote{A path between two points $\ub_1$ and $\ub_2$ is said to be {\it smooth} whenever it admits a parametrisation $\mathbf{r} : [0,1] \to \reals^d$ such that  $\mathbf{r}(0) = \ub_1$, $\mathbf{r}(1) = \ub_2$, and $t\mapsto\mathbf{r}(t)$ is  continuously differentiable.} path from $\ub_1$ to~$\ub_2$.

This interpretation is  underlying the definitions we now propose. Recall that~${\rm V}_{d-1}$ denotes the uniform distribution over the unit hypersphere $\mathcal S _{d-1}$ which, for~$d=1$,   reduces to the uniform discrete distribution ${\rm V}_0$ over $\mathcal S _{0}=\{-1, 1\}$.
\begin{definition} 
\label{Defn42bis}
(Lorenz potential functions)

\noindent
\ \ \  Assuming that they are well-defined, 
\begin{compactenum}
\item[(i)]{\rm 
call {\em (absolute) center-outward Lorenz~potential function of }$\Xb$ the map
\begin{align}
\label{eq: coDefLCPot}
u\mapsto \Lambda_{\Xb\pm}(u) &\coloneqq  \int_{{\mathcal S}_{d_\Xb-1}}\psi_{\Xb\pm}(u\sbold ) \diff {\rm V}_{d-1}(\sbold)=\E[\psi_{\Xb\pm}( u\Sb )], \quad u\in [0,1); 
\end{align}
}
\item[(ii)]  {\em assuming   $d_\Xb = d=d_\Yb$, call }  (absolute) center-outward Lorenz potential function of~$\Yb$ with respect to $\Xb$   {\em  the map}
\begin{align}
\label{eq: coDefLCPot'}
	u\mapsto \Lambda_{\Yb/\Xb\pm}(u)&\coloneqq   \int_{{\mathcal S}_{d -1}}\psi_{\Yb\pm}\left(\Fb_{\Yb\pm}\circ \Qb_{\Xb\pm}(u\sbold)\right) \diff {\rm V}_{d-1}(\sbold), \quad u\in [0,1).  
\end{align}
\end{compactenum}
\end{definition}

Contrary to the Lorenz functions in Definition~\ref{c-oLor2}, which take values in $\mathbb R^{d_\Xb}$ or~$\mathbb R^{d_\Yb}$,  Lorenz potential functions are real-valued.  

From a physical perspective,~$\Lambda_{\Xb\pm}(u)$ in~\eqref{eq: coDefLCPot} can be interpreted as the expected  work  it takes, in the field with potential $\psi$,  to move  a particle with unit mass  from zero to a random point  $\bf x$ uniformly distributed over the quantile contour of order $u$.

The  concept of Lorenz potential function nicely simplifies in the case of (nontrivial) spherical distributions. In that case, we have  
\[
\frac{\Lambda_{\Xb\pm}(u)}{\Lambda_{\Xb\pm}(1)} = \Ell_{\lVert \Xb\rVert}(u), \qquad u \in [0,1)
\]
The relative Lorenz potential function of a spherical random variable $\Xb$ thus reduces to the classical univariate Lorenz function for $\lVert \Xb\rVert$,  since~$ \psi_{\Xb\pm}(u\sbold )$ then only depends on~$u$. 

\section{Multivariate center-outward  concentration indices}
\label{SecGini}

\subsection{Center-outward    Gini  and Pietra concentration indices}\label{ginisec} Let us assume that all components of $\Xb$ (for the center-outward Lorenz function of $\Xb$) or $\Yb$ (for the center-outward Lorenz function of $\Yb$ with respect to $\Xb$) are positive-valued with finite means, so that the corresponding relative functions (see Section~\ref{incshsec}) make sense. The following indices then are natural generalizations of the traditional univariate ones.

\begin{definition}
\label{c-oLor2'} 
(Center-Outward Gini and Pietra indices)
\begin{compactenum}
\item[(i)]{\rm  Call   
\begin{align}
\label{eq: coDefGinimult}
G_{\Xb \pm} \coloneqq \frac{2}{\sqrt{d}} \int_{0}^1 \left\Vert (u,\ldots,u)^\prime - {\mathcal L}_{\Xb \pm}(u))
\right\Vert {\rm d}u
\end{align}
and 
\begin{align}
\label{eq: coDefPietramult}
P_{\Xb \pm} \coloneqq \frac{1}{\sqrt{d}} \sup_{0\leq u\leq 1} \left\Vert (u,\ldots,u)^\prime - {\mathcal L}_{\Xb \pm}(u))
\right\Vert
\end{align}
the {\em  Gini} and {\it Pietra  center-outward G-indices of} $\Xb$ 
 respectively;
}
\item[(ii)]{\rm  call %
\begin{align}
\label{eq: coDefGinimult'}
G_{\Yb/\Xb \pm} \coloneqq \frac{2}{\sqrt{d}}\int_{u=0}^1 \left\Vert (u,\ldots,u)^\prime - {\mathcal L}_{\Yb/\Xb \pm}(u)
\right\Vert {\rm d}u
\end{align}
and
\begin{align}
\label{eq: coDefPietramult'}
P_{\Yb/\Xb \pm} \coloneqq\frac{1}{\sqrt{d}} \sup_{0\leq u\leq 1} \left\Vert (u,\ldots,u)^\prime - {\mathcal L}_{\Yb/\Xb \pm}(u)
\right\Vert {\rm d}u
\end{align}
the {\em  Gini} and {\it Pietra center-outward G-indices of $\Yb$ with respect to $\Xb$}, 
respectively;
}
\item[(iii)] call 
\begin{align}
\label{eq: coDefGinimult''}
GK_{\Yb/\Xb \pm} \coloneqq \frac{2}{\sqrt{d}}\int_{u=0}^1 \left\Vert (u,\ldots,u)^\prime - {\mathcal K}_{\Yb/\Xb \pm}(u)
\right\Vert {\rm d}u
\end{align}
and
\begin{align}
\label{eq: coDefPietramult''}
PK_{\Yb/\Xb \pm} \coloneqq\frac{1}{\sqrt{d}} \sup_{0\leq u\leq 1} \left\Vert (u,\ldots,u)^\prime - {\mathcal K}_{\Yb/\Xb \pm}(u)
\right\Vert {\rm d}u,
\end{align}
the {\em  center-outward Gini} and  {\it Pietra K-indices of $\Yb$ with respect to $\Xb$},   
respectively. 
\end{compactenum}
\end{definition}

These measures quantify the discrepancy between the rescaled Lorenz and Kakwani curves and the diagonal of the corresponding unit cube.
The scaling factors in the definitions~(5.1)-(5.6) are chosen so that the indices take their values between zero and one. 
The quantities~(5.1)-(5.6) are maximal for ${\mathcal L}_{\Xb \pm}(u)$ (respectively, ${\mathcal L}_{\Yb / \Xb \pm}(u)$, ${\mathcal K}_{\Yb / \Xb \pm}(u)$) arbitrarily close to~$(0,\ldots, 0)^\prime$ for $0 \le u< 1$, to $(1,\ldots, 1)^\prime$ for $u= 1$. 

\subsection{Center-outward   Gini  and Pietra potential indices}
\label{ginisec}

Definition~\ref{Defn42bis} and the interpretation  of  Lorenz functions as potentials  naturally suggest   different indices.
\begin{definition}( Center-outward Gini and Pietra potential  concentration indices)
\label{dfn: GinPot}
\begin{compactenum}
\item[]{\it (i)} {\rm 
 Call 
\begin{equation}
\label{eq: coGini}
G^\Lambda_{\Xb\pm}\coloneqq 2 \int_0^1 \left [ u - \Lambda_{\Xb\pm}(u)/\Lambda_{\Xb\pm}(1) \right ] \diff u 
\end{equation}
and
\begin{align}
\label{eq: coPietra}
P^\Lambda_{\Xb\pm}\coloneqq  \sup_{0 \leq u \leq 1} \left[ u - \Lambda_{\Xb\pm}(u)/\Lambda_{\Xb\pm}(1) \right] 
\end{align}
the  {\em  center-outward Gini  and Pietra potential  concentration indices of } $\Xb$, respectively;
}
\item[(ii)]  {\em assuming   $d_\Xb = d=d_\Yb$, call} 
\begin{align}
\label{eq: coDefLCPot'}
G^\Lambda_{\Yb /\Xb}\coloneqq	2 \int_0^1 \left [ u - \Lambda_{\Yb/\Xb\pm}(u)/\Lambda_{\Yb/\Xb\pm}(1) \right] \diff u 
\end{align}
{\rm and}
\begin{align}
\label{eq: coDefLCPot'}
P^\Lambda_{\Yb /\Xb}\coloneqq	 \sup_{0\leq u\leq 1} \left[ u - \Lambda_{\Yb/\Xb\pm}(u)/\Lambda_{\Yb/\Xb\pm}(1) \right] ,\end{align}
{\rm the}  {\it center-outward Gini and Petra potential concentration indices  of~$\Yb$ with respect to $\Xb$,}  {\rm respectively. }
\end{compactenum}
\end{definition}

These indices take value zero 
 when  $\Lambda_{\Yb/\Xb\pm}(u)/\Lambda_{\Yb/\Xb\pm}(1)  = u$ for almost all~$0\leq~\!u\leq~\!1$,   that is, when $\mathbf{X}$  is an almost sure constant.  This, however, is not necessary: the same indices also vanish if, for instance,   the expected slope of the line joining $0$ and $ \psi(u \mathbf{S} )$,\linebreak  where~$\mathbf{S}\sim\rm{V}_{d-1} $,  does not depend on $u$.

\subsection{Center-outward quantile function and Gini mean difference }
 A global multivariate Gini index also can be constructed as an extension, based on  center-outward quantile functions,  of the characterization~\eqref{eq: GiniAlt} of the traditional concept.
\begin{definition}
\label{def: GKM}
Call 
\begin{equation}
\label{eq: GKM}
G^\mathrm{KM}_{\Xb\pm}\coloneqq  \frac{1}{2\kappa}\int\int \lVert {\bf Q}_{\Xb\pm}(\grx) - {\bf Q}_{\Xb\pm}(\gry) \rVert  \diff \rm{U}_d( \grx ) \diff \rm{U}_d(\gry),
\end{equation}
where 
$
\kappa \coloneqq  \int \lVert {\bf Q}_\pm(\grx) \rVert  \diff \mathrm{U}_d(\grx),
$
the {\it Koshevoy--Mosler multivariate center-outward Gini index of $\Xb$.}
\end{definition}

It is easy  to see that $G^\mathrm{KM}_{\Xb\pm }$ is similar to the ``Gini mean difference''  proposed and studied by \citet{koshevoy1997multivariate}. The introduction of the center-outward quantile functions provides a new interpretation of that coefficient. Indeed, the integral in \eqref{eq: GKM} is a measure of the variability of the transport plan ${\bf Q}_{{\bf X}\pm}$  in an $L^1(\mathrm{U}_d)$ sense. From a geometric point of view, the value of this index is an intrinsic property of elements of the Wasserstein tangent space at $\mathrm{U}_d$. 
Furthermore, that concept coincides, in dimension $d=1$, with the classical Gini coefficient    $G_X$. 

\begin{proposition} 
For a nonnegative real-valued random variable $X$ with $\expec X \neq 0$,   
\[
G^{\rm{KM}}_{X\pm} = G_X.
\]
\end{proposition}
\begin{proof}
This is a consequence of the change of variable formula and the relationships from Section~2.1. Indeed,
\begin{align*}
G_X &= \frac{1}{2\mu}\int\int\big\vert Q(t) - Q(s) \big\vert  \ind_{ \{0\le s \le 1\} }   \ind_{ \{0\le t \le 1\} }\diff s  \diff t \\
 & =\frac{1}{2\mu}\int\int \big\vert Q_\pm(t) - Q_\pm(s) \big\vert  \frac{\ind_{ \{-1\le s \le 1\} }} {2}  \frac{\ind_{ \{-1\le t \le 1\} } }{2}\diff s  \diff t = G^{\mathrm{KM}}_{X\pm}.
\end{align*} \vspace{-20.5mm}

 \hfill \qedhere
\end{proof}\medskip

\section{Estimation}
\label{sec: Est}

So far, only population concepts have been considered. In practice, one is dealing with observations
\begin{equation}
\label{eq; samplingMechanism}
\left( \Xb,\Yb\right)\n\!\!\coloneqq\! \left((\Xb\n_1,\Yb\n_1),\ldots,(\Xb\n_n,\Yb\n_n)\right), \quad n \in \NN .
\end{equation}
Whenever it can be assumed that these observations are an i.i.d.\ sample  with distribu\-tion~$\prob_{\Xb,\Yb}$, $\Xb_i\n \sim \prob_\Xb \in \Prob^{d_\Xb}$, and $\Yb_i\n\sim\prob_\Yb \in \Prob^{d_\Yb}$, the empirical counterparts (described in Section~\ref{sec: MultLoEst} below)  constitute  estimators of  the population Lorenz and Kakwani functions $L_{\Xb\pm}$,  $L_{\Yb/\Xb\pm}$,  $K_{\Yb/\Xb\pm}$, etc.\  of Definition~\ref{c-oLor2}; the consistency properties of these estimators are investigated in Section~\ref{sec: MultLoEst}.  Depending on the context, this assumption of an~i.i.d.\ sample,  however, may be unrealistic; the same estimators, then, should be considered from a purely descriptive point of view and consistency is meaningless.

In the remainder of this section, we thus make the assumption that \eqref{eq; samplingMechanism} is an~i.i.d.\ sample  and consider the problem of estimating  the  corresponding population Lorenz and Kakwani functions. 
 Whenever necessary (i.e., when dealing with a.s.\ convergence), we tacitly assume that  the sequence~$\{\left( \Xb,\Yb\right)\n$, $n\in{\mathbb N}\}$ is defined on a single probability space~$(\Omega, \mathcal{A}, \mathbb{P})$. For simpli\-city, the superscripts $^{(n)}$ are omitted when no confusion is possible.

\subsection{Empirical  distribution and quantile functions}
\label{eq: EmpRanks}
 The empirical counter\-parts~$\Fb_{\Xb\pm}\n$ of $\Fb_{\Xb\pm}$ and  $\Qb_{\Xb\pm}\n$ of $\Qb_{\Xb\pm}$ are obtained as the solution of an optimal matching problem between the sample values~$\Xb\n_1,\ldots,\Xb\n_n$ and a ``regular'' grid ${\mathfrak G}\n$ with gridpoints ${\scriptstyle{\mathfrak{G}}}\n_1,\ldots,{\scriptstyle{\mathfrak{G}}}\n_n$ such that the empirical distribution over~${\mathfrak G}\n$ converges weakly to ${\rm U}_d$. More precisely,  
	$\left(\Fb_{\Xb\pm}\n(\Xb\n_1),\ldots, \Fb_{\Xb\pm}\n(\Xb\n_n)\right)$ 
 is defined as the mini\-mizer~$\left(
{\scriptstyle{\mathfrak{G}}}\n_{\pi^*(1)},\ldots,{\scriptstyle{\mathfrak{G}}}\n_{\pi^*(n)}\right)$, over the $n!$ possible permuta\-tions~$\pi\in\Pi_n$ of the integers~$\{1,\ldots,n\}$, of the sum of squared Euclidean distances 
$
  \sum_{i=1}^n\big\Vert {\Xb}\n_{i} - {\scriptstyle{\mathfrak{G}}}\n_{\pi(i)})
\big\Vert ^2,
$
while $\Qb_{\Xb\pm}\n$ is characterized by  $\Qb_{\Xb\pm}\n({\scriptstyle{\mathfrak{G}}}\n_{\pi^*(i)})\coloneqq \Xb\n_i$, $i=1,\ldots,n$.
%

The empirical distribution  function $\Fb_{\Xb\pm}\n$  arising from this optimization procedure is  defined at the observations only, and the empirical quantile  function~$\Qb_{\Xb\pm}\n$ at the gridpoints. Both, however, can be extended into $\overline{\Fb}_{\Xb\pm}\n$ with domain $\mathbb R^d$ and~$\overline{\Qb}_{\Xb\pm}\n$ with domain $\mathbb S_d$ by means of an interpolation that preserves the properties (cyclical monotonicity) of the transport map (i.e., being the gradient of a convex potential). 
Indeed, as shown in \citet{del2018center}, there exist smooth and  Fenchel--Legendre conjugate empirical potentials      $\bar\psi_{\Xb\pm}^{* (n)}$ and~$\bar\psi_{\Xb\pm}^{(n)}$ such that 
$$\overline{\Fb}_{\Xb\pm}\n(\Xb _i)\coloneqq \nabla\bar\psi_{\Xb\pm}^{*(n)}(\Xb _i) = \Fb_{\Xb\pm}\n(\Xb _i)\quad\text{ and}\quad \overline{\Qb}_{\Xb\pm}\n(\Xb _i)\coloneqq \nabla\bar\psi_{\Xb\pm}^{(n)}(\Xb _i) = \Qb_{\Xb\pm}\n(\Xb _i)$$ for all $i=1,\ldots, n$; these potentials can be constructed from~$\Fb_{\Xb\pm}\n$ (or from $\Qb_{\Xb\pm}\n$ in a similar fashion) via the following {\it Yosida--Moreau regularization}.

For any realization $\grx\n$  of $\Xb\n$,  owing to the cyclical monotonicity property of optimal transport, there exists $\{\psi_{i,\pm}^*,\ i=1,\ldots, n\}$ such that 
\[
\grx_i\n \cdot {\scriptstyle{\mathfrak{G}}}\n_i -\psi_{i,\pm}^* \geq \max_{1\le j\le n} (\grx_i\n \cdot {\scriptstyle{\mathfrak{G}}}\n_j- \psi_{j,\pm}^* ) , \quad  i=1,\ldots, n
\]
and the equality is strict when the maximum is restricted to $j\neq i$.
Defining the convex map
\[
\grx \mapsto \phi^*( \grx)  = \max_{1\le j\le n} (\grx \cdot {\scriptstyle{\mathfrak{G}}}\n_j -\psi_{j,\pm}^* ) , \quad  i=1,\ldots, n,
\]
and its regularized version 
\begin{equation}
\label{eq: regPot} \bar\psi_{\Xb\pm}^{*(n)}(\grx )\coloneqq 
\phi_\epsilon( \grx)  \coloneqq \inf_{\gry\in {\mathfrak G}\n} \left[ \phi^*( \mathbf{y }) + \frac{1}{2\epsilon}  \lvert \mathbf{y }-\mathbf{x } \rvert^2  \right]
\end{equation}
where 
\[
\epsilon\coloneqq  \frac{1}{2} \min_i \big( (\grx_i \cdot {\scriptstyle{\mathfrak{G}}}\n_i -\psi_{i,\pm}^* ) - \max_{j\neq i} (\grx_i \cdot {\scriptstyle{\mathfrak{G}}}\n_j- \psi_{j,\pm}^* ) \big),
\]
the desired regularized empirical distribution function is $\overline{\Fb}_{\Xb\pm}\n\coloneqq \nabla\bar\psi_{\Xb\pm}^{*(n)}$, which moreover satisfies the Glivenko-Cantelli property 
\begin{equation}
\label{eq: GlivCant}
\lim_{n \to \infty} \sup_{\grx\in \reals^d} \lVert \bar{\Fb}_{\Xb\pm}^{(n)}(\grx)-  \Fb_{\Xb\pm}(\grx) \rVert= 0 \quad \text{${\rm P}_{\Xb}${\rm -a.s.} }
\end{equation}
As above, one can also estimate the potential corresponding to the center-outward quantile function. For any realisation $\grx^{(n)}$  of $\Xb^{(n)}$ there exists an $n$-tuple~$\{\psi_{i,\pm},\ i=1,\ldots, n\}$ of real numbers such that 
\[
\grx_i\n \cdot {\scriptstyle{\mathfrak{G}}}\n_i -\psi_{i,\pm} \geq \max_{1\le j\le n} (\grx_j\n \cdot {\scriptstyle{\mathfrak{G}}}\n_i- \psi_{j,\pm} ) , \quad  i=1,\ldots, n,
\]
where the inequality is strict for $j\neq i$;  define $\hat{\psi}\n_{\Xb}$ as 
\begin{equation}\label{hatpsidef}
\gru \mapsto \hat{\psi}\n_{\Xb}( \gru)  \coloneqq \max_{1\le j\le n} (\grx_j\n \cdot \gru - \psi_{j,\pm} ) , \quad  i=1,\ldots, n.
\end{equation}
Similar interpolations $\mathbf{F}_{\Yb\pm}\n, \bar{\mathbf{F}}_{\Yb\pm}\n, \mathbf{Q}_{\Yb\pm}\n$, and $\bar{\mathbf{Q}}_{\Yb\pm}\n$ can be constructed for ${\mathbf{F}}_{\Yb\pm}$ and $\mathbf{Q}_{\Yb\pm}$; we refer to \citet{del2018center} for details. 

\subsection{Estimation of center-outward Lorenz and Kakwani  functions}
\label{sec: MultLoEst}
As estimators of  the Lorenz and Kakwani functions $L_{\Xb\pm}$, $L_{\Yb/\Xb\pm}$, and $K_{\Yb/\Xb\pm}$ of Definition~\ref{c-oLor2}, we propose the statistics 
\begin{compactenum}
\item[(a)]
 $\displaystyle{
u \mapsto  {\hat L}\n_{\Xb\pm}(u)\coloneqq  \frac1n\sum_{i=1}^n \Xb_i \ind{\{ \lVert \Fb_{\Xb,\pm}^{\n} (\Xb_i)  \rVert\le u  \}},
}
$ 
\item[(b)]
$
\displaystyle{u \mapsto  {\hat L}\n_{\Yb/\Xb\pm}(u)\coloneqq \frac1n\sum_{i=1}^n \Yb_i \ind{\{ \lVert \bar{\Fb}_{\Xb,\pm}^{\n} (\Yb_i)  \rVert\le u  \}}
 ,}
$ {\rm and}  
\item[(c)]
$
\displaystyle{u \mapsto  {\hat K}\n_{\Yb/\Xb\pm}(u)\coloneqq \frac1n\sum_{i=1}^n \Yb_i \ind{\{ \lVert \Fb_{\Xb,\pm}^{\n} (\Xb_i)  \rVert\le u  \}}
 }${\rm , respectively $(0\leq u\leq 1)$.}
\end{compactenum}

Remark that, replacing ${\Fb}_{\Xb,\pm}^{\n}$ by $\bar{\Fb}_{\Xb,\pm}^{\n}$  in \textit{(a)} and \textit{(c)}  yields  exactly the  same estimators: computing 
 $\bar{\Fb}_{\Xb,\pm}^{\n}$, thus, is not necessary.
For \textit{(b)}, quite on the contrary,  computing $\bar{\Fb}_{\Xb,\pm}^{\n}$ is required as ${\Fb}_{\Xb,\pm}^{\n}$ is not defined at the $\Yb_i$'s.

As estimators of the conditional counterparts of these curves (Definition~\ref{c-oLorCo}), we propose 
\begin{compactenum}
\item[(a)]
 $\displaystyle{
u \mapsto  {\tilde L}\n_{\Xb\pm}(u)\coloneqq  \frac{n}{n_0 + n_S \lfloor u (n_R+1) \rfloor } {\hat L}\n_{\Xb\pm}(u),
}
$\smallskip 
\item[(b)]
$
\displaystyle{u \mapsto  {\tilde L}\n_{\Yb/\Xb\pm}(u)\coloneqq \frac{n}{n_0 + n_S \lfloor u (n_R+1) \rfloor }{\hat L}\n_{\Yb/\Xb\pm}(u) ,}
$ {\rm and}\smallskip  
\item[(c)]
$
\displaystyle{u \mapsto  {\tilde K}\n_{\Yb/\Xb\pm}(u)\coloneqq \frac{n}{n_0 + n_S \lfloor u (n_R+1) \rfloor }{\hat K}\n_{\Yb/\Xb\pm}(u)
 }${\rm , respectively $(0\leq u\leq 1)$.}\medskip
\end{compactenum}

We then have the following consistency results.

\begin{proposition}
\label{prop: Consist}
Let $(\Xb, \Yb)^{(n)}$ be as in~\eqref{eq; samplingMechanism}. 
Assuming that $\expec \lVert \Xb\rVert $ and $\expec \lVert \Yb\rVert $ exist and are finite,  
\begin{compactenum}
\item[(I)]${\hat L}\n_{\Xb\pm}$,  ${\hat L}\n_{\Yb/\Xb\pm}$, and ${\hat K}\n_{\Yb/\Xb\pm}$  are uniformly (in  $u \in [0,1]$) strongly consistent estimators of ${L}_{\Xb\pm}$,  ${L}_{\Yb/\Xb\pm}$, and ${K}_{\Yb/\Xb\pm}$, respectively;
\item[(ii)] for any fixed $\eps >0$,   ${\tilde L}\n_{\Xb\pm}(u),{\tilde L}\n_{\Yb/\Xb\pm}(u)$, and $ {\tilde K}\n_{\Yb/\Xb\pm}(u)$ are uniformly (in  $u \in [\eps, 1]$)  strongly  consistent estimators of ${L}^{\text{\tiny cond}}_{\Xb\pm}$,  ${L}^{\text{\tiny cond}}_{\Yb/\Xb\pm}$, and ${K}^{\text{\tiny cond}}_{\Yb/\Xb\pm}$, respectively.
\end{compactenum}
\end{proposition}
\begin{proof}
(i) It follows from the Glivenko--Cantelli result \eqref{eq: GlivCant} that, for all $i$,  
 \[
 \ind_{\left\{ \lVert{\Fb}_{\Xb,\pm}^{\n} (\Xb_i)\rVert \le u \right \}} = \ind_{ \left\{ \lVert {\Fb}_{\Xb,\pm}^{ {\color{white} \n}} (\Xb_i)  \rVert \le u \right \}} + o_{\text{\rm a.s.}}(1) \]
 and  
$$
 \ind_{\left\{ \lVert \bar{\Fb}_{\Xb,\pm}^{\n} (\Yb_i)\rVert \le u \right\}} = \ind_{\left\{ \lVert \bar{\Fb}_{\Xb,\pm}^{ {\color{white} \n}} (\Yb_i)  \rVert \le u  \right\}} +  o_{\text{\rm a.s.}}(1) 
$$
as $n\to\infty$, uniformly in $u$. 
Since $\expec \lVert \Yb\rVert <\infty$ entails that $  o_{\text{\rm a.s.}}(n^{-1})  \sum_{i=1}^n\Yb_i$ is  $o_{\text{\rm a.s.}}(1) $, 
\[
 \frac1n\sum_{i=1}^n\left[\Yb_i \ind_{\left\{ \lVert\Fbpm^{\n} (\Xb_i) \rVert\le r \right \}}\right] =  \frac1n\sum_{i=1}^n\left[\Yb_i \ind_{\left\{ \lVert\Fbpm (\Xb_i) \rVert \le r  \right\}}\right] + o_{\text{\rm a.s.}}(1) .
\]
Further,  the class of functions
 $
 \{(x,y)\mapsto y \ind_{\{\Fbpm(x)\le r\}} \}_{r \in [0,1]} 
$ 
 is pointwise measurable and has a finite bracketing number. Hence, it is Glivenko--Cantelli, and 
\begin{align*}
\lim_{n \to \infty} \frac1n\sum_{i=1}^n \Yb_i \ind_{\left\{ \lVert\Fbpm^{\n} (\Xb_i) \rVert\le r \right \}} 
&= \lim_{n \to \infty} \frac1n\sum_{i=1}^n \Yb_i \ind_{\{ \lVert\Fbpm (\Xb_i) \rVert\le r  \}} \\
&= \expec \big(\Yb \ind_{\{ \lVert\Fbpm (\Xb) \rVert \le r  \}}\big),
\end{align*}
uniformly in $r$.  The claim for ${\hat K}\n_{\Yb/\Xb\pm}$ follows. 
The proof for ${\hat L}\n_{\Xb\pm}$  and ${\hat L}\n_{\Yb/\Xb\pm}$ follows along the same lines.

Turning to (ii), note that ${\tilde L}\n_{\Xb\pm}$, ${\tilde L}\n_{\Yb/\Xb\pm}$, and ${\tilde K}\n_{\Yb/\Xb\pm}$  all are obtained as a rescaling  
by~$n/(n_S \lfloor (n_R+1)u\rfloor)$ of the estimators of Proposition~\ref{prop: Consist}.  The claim follows from the fact that   these estimators   and the rescaling function are bounded and  {strongly} converge,  uniformly in $u\in [\eps, 1]$ for all $\eps >0$. 
\end{proof}

\begin{remark}
Further asymptotic results such as asymptotic distributions for the estimators in Proposition~\ref{prop: Consist}, 
 of course, are highly desirable.  
 Such results, however, are much beyond the scope of the present work. Although the study of the asymptotic behavior of (functions of) empirical optimal transport plans and potentials is a very active area of research, many challenging questions remain unsolved; 
we refer to \citet{gunsilius2018convergence, hutter2021minimax}  and the references therein for an overview of some of the results currently available. 
\end{remark}

For the estimation of the Lorenz potential functions, we will require  the  $n$-point grid~${\mathfrak G}\n$ to exhibit a specific nature. 
Factorizing $n$ into $n=n_Rn_S + n_0$ with   remainder\linebreak  term~$n_0 \le \min(n_R,n_S)$, construct~${\mathfrak G}\n$ as the outer product\footnote{ This amounts to considering the grid consisting of the union of the points of the $n_R$ spherical grids multiplied by $k/(n_R+1), 1\le k\le n_R$.} of the radial grid
\begin{equation}\label{radgrid}\{1/(n_R+1), 2/(n_R+1),\ldots,  n_R/(n_R+1)\}
\end{equation}
 over $[0,1]$ and a uniform grid over ${\mathcal S}_{d-1}$ such that the discrete uniform measure over these gridpoints converges weakly  as $n_S \to \infty$ to the uniform measure ${\rm V}_{d-1}$ on $\mathcal{S}_{d-1}$, then add~$n_0$ copies of the origin to this grid of $n_Sn_R$ points. Call this a {\it regular} grid.

For the Lorenz potential  functions of Definition~\ref{Defn42bis},  we propose the estimators
\begin{compactenum}
\item[(a)]
 $ \displaystyle{
u \mapsto \hat \Lambda\n_{\Xb\pm}(u) \coloneqq \frac1{n_S}\sum_{i \in \mathcal{S}\n_u }  \hat{\psi}\n_{\Xb \pm}( {\scriptstyle{\mathfrak{G}}}\n_i)
}$
  and 
\item[(b)]
 $ \displaystyle{
u \mapsto \hat\Lambda\n_{\Yb/\Xb\pm}(u) \coloneqq
\frac1{n_S}\sum_{i \in \mathcal{S}\n_u }  \hat{\psi}\n_{\Yb \pm}( \bar{\Fb}^{\n}_{\Yb\pm} (\Qb_{\Xb,\pm}^{(n)}({\scriptstyle{\mathfrak{G}}}\n_i)),
}$
\end{compactenum}
with $\hat{\psi}\n_{\Xb \pm}$, $\hat{\psi}\n_{\Yb \pm}$, $\bar{\Fb}^{\n}_{\Yb\pm}$, and $\Qb_{\Xb,\pm}^{(n)}$ as in Section~\ref{eq: EmpRanks} and $ \mathcal{S}\n_u\!\! \coloneqq\! \left\{\! i : \lVert {\scriptstyle{\mathfrak{G}}}\n_i \rVert\! =\! \frac{ \lfloor  (n_R+1) u \rfloor}{ n_R+1}  \!\right\}$. 


\begin{proposition}
\label{prop: EstPotUnif}
Consider a {\em regular} grid ${\mathfrak G}\n$ constructed as above. Then, 
 for any fixed~$u \in [0,1)$,  $\hat\Lambda\n_{\Yb/\Xb\pm}$ and  $ \hat\Lambda\n_{\Yb/\Xb\pm}$ are strongly consistent  estimators of~$\Lambda_{\Xb\pm}$ and~$\Lambda_{\Yb/\Xb\pm}$, respectively, uniformly over  $[0,1-\epsilon]$  for any $\epsilon >0$. 
\end{proposition}

\begin{proof}
We actually only need to prove the second part of the claim since, for any fixed~$u<1$, there exists an $\epsilon>0$ such that $u \in [0,1-\epsilon]$.

We start by proving that the empirical potentials a.s.\ converge  to their population counterparts uniformly on the closed ball $\widebar{(1-\epsilon)\ball^d}$.  
The claim then follows from the weak convergence to ${\rm U}_d$ of the empirical distribution over the regular grid ${\mathfrak G}\n$.

By the almost sure convergence of empirical measures (see, e.g., \cite{varadarajan58}), 
 the sequences  $\prob_{\Xb}\n(\omega)$ of empirical measures over $\Xb_1,\ldots,\Xb_n$ weakly converges to ${\rm P}_\Xb$ for any~$\omega \in \Omega_0 \subseteq \Omega$ where $\Omega_0$ has probability one.  Denoting by ${\rm P}\n_{{\mathfrak G}}$ the empirical distribution over ${\mathfrak G}\n$, ${\rm P}\n_{{\mathfrak G}}$ similarly weakly converges to ${\rm U}_d$. 
Using \citet[Theorem~5.20]{villani2009optimal} and the uniqueness  of the optimal transport plan  for $\prob_\Xb\in {\mathcal P}_d$, it follows from a subsequence argument that the sequence of optimal transport plans from~$\prob_{\Xb}\n(\omega)$ to~${\rm P}\n_{{\mathfrak G}}$  converges weakly to the optimal transport plan from~$\prob_\Xb$ to $\rm{U}_d$.  
As  explained in \citet[proof of Proposition~3.3]{del2018center}, one can deduce from that weak convergence and~\citet[Theorem~2.8]{del2019central} the  pointwise convergence, for each $\omega \in \Omega_0$, of the corresponding empirical potentials. Now---see the proof of Theorem~1.52 in \citet{santambrogio2015optimal}---the empirical optimal potentials have the same modulus of continuity as the cost function, here the squared Euclidean distance. Since a sequence of real-valued equicontinuous functions converging pointwise on a compact set converges uniformly over that set,  the above convergence of potentials is uniform in ${\bf u}\in \widebar{(1-\epsilon)\ball^d}$.

Since  the grids ${\mathfrak G}\n$ are constructed as the product of a radial grid and a uniform grid over ${\mathcal S}_{d-1}$, both converging in distribution to  uniform measures, over $[0,1]$ and  $\mathcal{S}_{d-1}$, respectively,   the measures 
$
\frac1{n_S}\sum_{i \in \mathcal{S}\n_u } \delta_{\scriptstyle{\mathfrak{G}}\n_i}
$ 
converge weakly to  the uniform measure ${\rm V}_{d-1}^u$ on $u{\mathcal S}_{d-1}$ uniformly in $u$. To see this, note that~$
\big\vert [ (n_R+1)u] -(n_R+1)u \big\vert /(n_R+1)
$ converges to zero uniformly in $u$.
This uniform weak convergence,  combined with the uniform convergence of the potential over the closed ball $\widebar{(1-\epsilon)\ball^d}$, ensures that the result holds for~$\hat\Lambda\n_{\Xb\pm}$.

The result for $\hat\Lambda\n_{\Yb/\Xb\pm}$  follows essentially from the same argument, exploiting further the continuity of $\psi_{\Yb \pm}$ and $\Fb_{\Yb\pm}$ as well as the Glivenko--Cantelli result from Equation~\eqref{eq: GlivCant}. 
Further, one also requires the uniform convergence of $\Qb_{\Xb,\pm}^{(n)}$ over  $\widebar{(1-\epsilon)\ball^d}$: this  follows from \citet[Theorem~1.1]{segers2022graphical}.
\end{proof}

Interestingly, this result would not necessarily hold for arbitrary grids converging weakly to ${\rm U}_d$, and the special ``product structure" form of the grid is required here as one needs the empirical measure underlying the empirical contour of level $u$ to converge weakly to the uniform on $u\sphere_{d-1}$.


\subsection{Estimation of  multivariate Gini  and Pietra indices} 

Define
\begin{align}\nonumber
&\hat{\mathcal L}_{\Xb \pm}(u)\coloneqq {\hat L\n_{\Xb\pm}(u)} \oslash \frac1n \sum_{i=1}^n \mathbf{X}_i,  
&\hat{\mathcal L}_{\Xb \pm}^{\text{\tiny cond}}(u) \coloneqq {\tilde L\n_{\Xb\pm}(u)} \oslash \frac1n \sum_{i=1}^n \mathbf{X}_i, \qquad\\
&\hat{\mathcal L}_{\Yb/\Xb \pm}(u) \coloneqq {\hat L\n_{\Yb/\Xb\pm}(u)} \oslash \frac1n \sum_{i=1}^n \mathbf{Y}_i,   
&\hat{\mathcal L}_{\Yb/\Xb \pm}^{\text{\tiny cond}}(u) \coloneqq {\tilde L\n_{\Yb/\Xb\pm}(u)} \oslash \frac1n \sum_{i=1}^n \mathbf{Y}_i, 
\nonumber \\
&\hat{\mathcal K}_{\Yb/\Xb \pm}(u) \coloneqq {\hat K\n_{\Yb/\Xb\pm}(u)} \oslash \frac1n \sum_{i=1}^n \mathbf{Y}_i, \quad \text{and }  
&\hat{\mathcal K}_{\Yb/\Xb \pm}^{\text{\tiny cond}}(u) \coloneqq {\tilde K\n_{\Yb/\Xb\pm}(u)} \oslash \frac1n \sum_{i=1}^n \mathbf{Y}_i.
\vspace{-4mm}\label{**}\end{align}
where $\oslash$ denotes Kronecker/componentwise division. The Gini  and Pietra indices  of Definition~\ref{c-oLor2'} easily can be estimated by repla\-cing, in equa\-tions~\eqref{eq: coDefGinimult}--\eqref{eq: coDefPietramult''},~${\mathcal L}_{\Xb \pm}(u)$, ${\mathcal L}_{\Yb/\Xb \pm}(u)$, and ${\mathcal K}_{\Yb/\Xb \pm}(u)$ 
with  $\hat{\mathcal L}_{\Xb \pm}(u)$,~$\hat{\mathcal L}_{\Yb/\Xb \pm}(u)$, and~$\hat{\mathcal K}_{\Yb/\Xb \pm}(u)$. 
\begin{proposition} 
For any $u \in (0,1)$, the estimators of the indices defined in~\eqref{eq: coDefGinimult}--\eqref{eq: coDefPietramult''} 
obtained by substituting,  in these definitions, ${\mathcal L}_{\Xb \pm}(u)$, ${\mathcal L}_{\Yb/\Xb \pm}(u)$, and ${\mathcal K}_{\Yb/\Xb \pm}(u)$ with their empirical counterparts  
are strongly consistent.  
\end{proposition}
\begin{proof}
The proof follows directly from Proposition~\ref{prop: Consist}, the strong law of large numbers, Slutsky's theorem, and the continuous mapping theorem.
\end{proof}

Now, let us turn to the estimation of the concepts of   Gini and Pietra potential concentration indices proposed in Definition~\ref{dfn: GinPot}. 
\begin{proposition} 
\label{prop: GinPotEst}
Denote by $\hat{G}^\Lambda_{\Xb\pm}$, $\hat{P}^\Lambda_{\Xb\pm}$, $\hat{G}^\Lambda_{\Yb/\Xb\pm}$, and $\hat{P}^\Lambda_{\Yb/\Xb\pm}$  the empirical counterparts of the potential concentration indices of Definition~\ref{dfn: GinPot}, obtained by plugging-in 
$\hat{\Lambda}_{\Xb\pm}$ and $\hat{\Lambda}_{\Yb/\Xb\pm}$ instead of ${\Lambda}_{\Xb\pm}$ and~${\Lambda}_{\Yb/\Xb\pm}$ respectively, in~\eqref{eq: coGini}--\eqref{eq: coDefLCPot'}. 
Then, assuming that $\supp(\prob_\Xb)$ is compact, 
$\hat{G}^\Lambda_{\Xb\pm}$ and~$\hat{P}^\Lambda_{\Xb\pm}$ are strongly consistent estimators of ${G}^\Lambda_{\Xb\pm}$ and ${P}^\Lambda_{\Xb\pm}$, respectively.
If, furthermore,   $\supp(\prob_\Yb)$ is compact,  $\hat{G}^\Lambda_{\Yb/\Xb\pm}$ and $\hat{P}^\Lambda_{\Yb/\Xb\pm}$ are strongly consistent estimators of ${G}^\Lambda_{\Yb/\Xb\pm}$ and ${P}^\Lambda_{\Yb/\Xb\pm}$, respectively. 
\end{proposition}
\begin{remark}
The crucial  point,  when proving the consistency of Gini and Pietra potential concentration indices, 
 is   the  uniform convergence of the empirical potentials  to the actual ones, which in turn guarantees that $\hat{\Lambda}_{\Xb\pm}(u)$ and~$\hat{\Lambda}_{\Yb/\Xb\pm}(u)$   converge to $\Lambda_{\Xb\pm} (u)$ and ${\Lambda}_{\Yb/\Xb\pm}(u)$  in a sufficiently strong sense for $\hat{G}^\Lambda_{\Xb\pm}$ and~$\hat{P}^\Lambda_{\Xb\pm}$ (respectively, $\hat{G}^\Lambda_{\Yb/\Xb\pm}$ and~$\hat{P}^\Lambda_{\Yb/\Xb\pm}$) to be consistent. The assumption of a compactly supported $\Xb$ (respectively, compactly supported $\Xb$ and~$\Yb$) entails such a convergence. Compactness here  is sufficient but certainly not necessary, and we expect the results of Proposition~\ref{prop: GinPotEst} to hold in more general situations. Relaxing compactness, however, would require a better understanding of the   behavior of  empirical potentials in the neighborhood of the unit sphere. This behavior is related to the boundary behavior of solutions of  the Monge--Amp\`ere equation---a highly nontrivial  problem in the noncompact case, see \cite{segers2022graphical}.
\end{remark}
\begin{proof}[Proof of Proposition~\ref{prop: GinPotEst}]
Owing to compactness, almost sure uniform convergence of the poten\-tials~$\psi_{\Xb\pm}$ and $\psi_{\Yb\pm}$ follows from Theorem~1.52 in \citet{santambrogio2015optimal}. The uniform convergence of the empirical quantile function follows again from  from \citet[Theorem~1.1]{segers2022graphical}.
Therefore, because of the more stringent compactness assumption, the convergence of the estimators of Proposition~\ref{prop: EstPotUnif} holds uniformly in $u \in [0,1]$. The claim then follows from a trivial application of Slutsky's lemma and the continuous mapping theorem.
\end{proof}

Finally consider the problem of estimating the Koshevoy--Mosler multivariate center-outward Gini index $G^\mathrm{KM}_{\Xb\pm}$ introduced in Definition~\ref{def: GKM}. Denote by $\Delta$ the double integral appearing in the numerator of \eqref{eq: GKM}. 
A natural estimator of $\Delta$ is 
\[
\hat{\Delta}\n\coloneqq  \frac{2}{n(n-1)}\sum_{1\leq i < j\leq n } \lVert \mathbf{Q}_{\Xb\pm}^{(n)}({\scriptstyle{\mathfrak{G}}}\n_i)-  \mathbf{Q}_{\Xb\pm}^{(n)}({\scriptstyle{\mathfrak{G}}}\n_j )\rVert =\frac{2}{n(n-1)}\sum_{1\leq i < j \leq n} \lVert \Xb_i -\Xb_j\rVert.
\]
This quantity is a U-statistic with kernel $h(\grx,\gry)= \lVert \grx -\gry \rVert$. The following result thus readily follows from, e.g.,   Theorem~12.3 in \cite{vaart_1998}.
\begin{proposition}
Let $\Xb_1, \Xb_2$, and $\Xb_2'$ be i.i.d.\ random variables with common distribution~$\prob_\Xb$.
If $\expec \lVert \Xb \rVert^2 < \infty$, then 
\[
\sqrt{n}\big(\hat{\Delta}\n- \expec \lVert \Xb_1 - \Xb_2 \rVert\big) \rightsquigarrow \normal(0, \sigma^2) \quad \text{as } n \to \infty,
\]
where $\sigma^2=4 \cov( \lVert \Xb_1 - \Xb_2 \rVert , \lVert \Xb_1 - \Xb_2' \rVert)$.
\end{proposition}
Based on this, for any consistent estimator $\hat{\kappa}$ of $\kappa \coloneqq  \int \lVert {\bf Q}_\pm(\gru) \rVert  \diff \mathrm{U}_d(\gru)$,  it follows from Slutsky's Lemma that 
\[
\sqrt{n}\left(\frac{\hat{\Delta}\n}{\hat{\kappa}} - G^\mathrm{KM}_{\Xb\pm} \right) \rightsquigarrow \normal\left(0, \frac{\sigma^2}{\kappa^2}\right) \quad \text{as } n \to \infty.
\]
A natural choice for $\hat{\kappa}$ is $\sum_{i=1}^n \lVert \Xb_i \rVert /n$.

\section{Empirical examples}
\subsection{Socio-economic inequalities and well-being }
\label{sec: App}
In this entire section, the curves we are depicting are piecewise constant over the interval~$u\in [0, n_R/(n_R +1)]$. For the sake of readability, however, we  are showing  linear interpolations of their points of discontinuity. 
For the same reason, $u$ is rescaled by $(n_R+1)/n_R$, so that the curves reach the top corner of the unit square or cube. 
\subsubsection{The dataset}
The relation between socio-economic inequalities and well-being is a fundamental topic in economics  and social sciences.  \citet{sarabia2013modeling} recently proposed a study of well-being across 132 countries  using the generalized Lorenz curve introduced in~~\citep{arnold2018majorization}, see~\eqref{eq: ArnoldBiv}. 

Well-being, of course, is the resultant of a number of socio-economic variables. Sarabia and Jorda are basing their  analysis  on  education,  income, and health, as measured by the indicators used by the United Nations  in the construction of the Human Development Index \citep[see, for instance, the latest report by the][]{UNDP}: Education~($X_1$) is based on the average between the Mean Years of Schooling   and the Expected Years of Schooling indices, Income~($X_2$) on a transformations of Gross National Income per capita\footnote{Gross National Income per capita is measured in Purchasing Power Parity international dollars. The Index we consider is further slightly different from that proposed by the UN. We did not apply the logarithmic transformation as it modifies the distribution of interest and, in particular, changes the tails behavior.}, and Health ($Y$) on  life expectancy at birth. Each of these variables  takes values between zero and one by construction.

In their application, \citet{sarabia2013modeling} consider, in dimension two, a parametric version of their Lorenz curve concept~\eqref{eq: ArnoldBiv},  assuming bivariate  Sarmanov--Lee distributions (a form of multivariate beta distribution, with beta  marginals). This requires,  in dimension two, the estimation of five parameters---two per marginal plus the dependence parameter. They repeat their analysis for each five-year period between 1980 and 2000. As an illustration  of our concepts,  we apply our methods (which do not require the somewhat arbitrary choice of a specific distribution, nor the estimation of its parameters)  to the most recent versions of the same three variables, collected  from the  United Nations website. 
 From that dataset, we only retained the $144$ countries for which data were available since 1990 and (in line with the quantisation  suggested in Table~1 of \citet{hallin2021finite})  factorized~$n=144$ into $n=n_Rn_S + n_0$ with $n_R=8$, $n_S=18$, and~$n_0=0$. 
 
The assumption of an~i.i.d.~sample  seems hard to justify for this dataset, though, and we feel that the analysis below, as well as the one by  Sarabia and Jorda (despite their Sarmanov--Lee distributional assumption),  should be taken from a purely descriptive perspective.

\subsubsection{Center-outward quantiles}
As a first step, let us show how the proposed center-outward approach determines empirical quantiles. 
Figure~\ref{fig: visualisation} displays the quantile  obtained from the empirical optimal transport map between the observations and the spherical grid. In the upper rightmost corner are Singapore, Switzerland, and Luxemburg, while the Central African Republic  is the lower leftmost point. The most central contours   consist of Egypt, Brazil or Mexico, to mention but three.  A positive interrelation between the two indices is quite visible. An interactive version of the plot enabling to select certain empirical center-outward quantile contours and providing additional information when hovering over the points is provided as digital supplemental material.

\begin{figure}[h!]
\begin{center}
\includegraphics[width=0.5\textwidth]{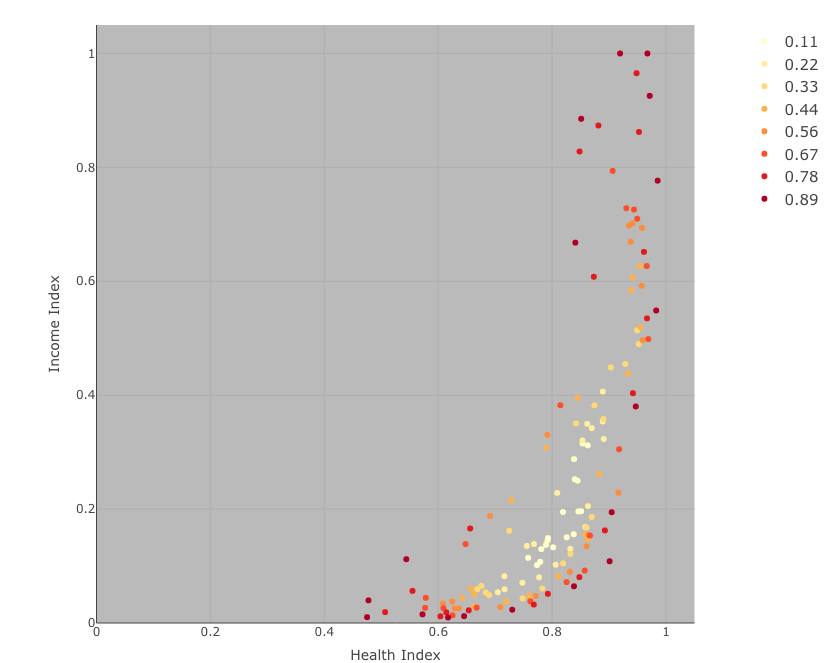}
\end{center}
\caption[]{\label{fig: visualisation} \slshape\small Center-outward quantiles for the Health (Life Expectancy)  and Income indices ($n=144$ countries in 2015). Each dot represents a country. The dots are coloured according to their center-outward quantile orders (their center-outward ranks, rescaled between~$1/(n_R+1)$ and~$n_R/(n_R+1)$). }\end{figure}

\subsubsection{Center-outward Lorenz and Kakwani  functions, Gini and Pietra coefficients.}

In this section, we are dealing with the empirical versions $\hat{\mathcal{L}}\n_{\Xb\pm}$,   $\hat{\mathcal{L}}\n_{Y/\Xb\pm}$, and    $\hat{\mathcal{K}}\n_{Y/\Xb\pm}$   of the concepts introduced in Definitions~\ref{c-oLor2} and~\ref{c-oLorCo}. 

\begin{figure}[h!]
\begin{tikzpicture}[scale=0.38]
\begin{axis} [
view = {70}{20},
legend pos=north west, 
 xmin=0, xmax=1,
    ymin=0, ymax=1,
        zmin=0, zmax=1,
xtick={0,0.5,1},
ytick={0,0.5,1},
ztick={0,0.5,1},
xlabel={$u (n_R+1)/n_R$},
ylabel={$X_1$},
zlabel={$X_2$},
scale=2,
3d box= complete*,
unit vector ratio*=2 2 2
]
 
\addplot3[color= blue] table 
{
x y   z
 0.000 0.0000000 0.00000000
 0.125 0.1292111 0.08083609
 0.250 0.2570173 0.17095984
 0.375 0.3846942 0.27025293
 0.500 0.5109293 0.38250466
 0.625 0.6359726 0.50827747
 0.750 0.7596859 0.64413090
 0.875 0.8798187 0.80917650
 1.000 1.0000000 1.00000000
};
 
 \addlegendentry{1990}
 
\addplot3[color= red] table 
{
x y   z
0.000 0.0000000 0.00000000
0.125 0.1255585 0.08475008
0.250 0.2531941 0.18034641
0.375 0.3800711 0.28945454
 0.500 0.5070930 0.40978241
 0.625 0.6332372 0.53922015
 0.750 0.7572982 0.67623918
 0.875 0.8786873 0.82455235
1.000 1.0000000 1.00000000
};
 
 \addlegendentry{2020}
 
\addplot3[color=gray]
 table 
{
x y   z
0.000 0.0000000 0.0000000
1 1.0000000 1.0000000
};
\end{axis}
\end{tikzpicture}
\begin{tikzpicture}[scale=0.55]
\begin{axis}[
    ylabel={$\hat{\mathcal{L}}_{\Xb\pm}(u)$  [$X_1$  component]},
    xlabel={$u (n_R+1)/n_R$},
    xmin=0, xmax=1,
    ymin=0, ymax=1,
    legend pos=north west,
        unit vector ratio*=1 1 
]
\addplot[
    color=blue
        ]
    coordinates {
 (0.000, 0.0000000 )
 (0.125, 0.1292111 )
 (0.250, 0.2570173 )
 (0.375, 0.3846942 )
 (0.500, 0.5109293 )
 (0.625, 0.6359726 )
 (0.750, 0.7596859 )
 (0.875, 0.8798187 )
 (1.000, 1.0000000 )
    };
 \addlegendentry{1990}
\addplot[
    color=red
        ]
    coordinates {
(0.000, 0.0000000 )
(0.125, 0.1255585 )
(0.250, 0.2531941 )
(0.375, 0.3800711 )
( 0.500, 0.5070930 )
( 0.625, 0.6332372 )
( 0.750, 0.7572982 )
( 0.875, 0.8786873 )
(1.000, 1.0000000 )
    };
 \addlegendentry{2020}
 \addplot[
    color=gray
        ]
    coordinates {
(0.000, 0.0000000)
(1, 1.0000000)
    };
\end{axis}
\end{tikzpicture}
\begin{tikzpicture}[scale=0.55]
\begin{axis}[
    xlabel={$u (n_R+1)/n_R$},
    ylabel={$\hat{\mathcal{L}}_{\Xb\pm}(u)$ [$X_2$ component]},
    xmin=0, xmax=1,
    ymin=0, ymax=1,
    legend pos=north west,
    unit vector ratio*=1 1 
]
\addplot[
    color=blue
        ]
    coordinates {
( 0.000 , 0.00000000)
( 0.125 , 0.08083609)
( 0.250 , 0.17095984)
( 0.375 , 0.27025293)
( 0.500 , 0.38250466)
( 0.625 , 0.50827747)
( 0.750 , 0.64413090)
( 0.875 , 0.80917650)
( 1.000 , 1.00000000)
    };
 \addlegendentry{1990}
\addplot[
    color=red
        ]
    coordinates {
(0.000 , 0.00000000)
(0.125 , 0.08475008)
(0.250 , 0.18034641)
(0.375 , 0.28945454)
( 0.500 , 0.40978241)
( 0.625 , 0.53922015)
( 0.750 , 0.67623918)
( 0.875 , 0.82455235)
(1.000 , 1.00000000)
    };
 \addlegendentry{2020}
 \addplot[
    color=gray
        ]
    coordinates {
(0.000, 0.0000000)
(1, 1.0000000)
    };
\end{axis}
\end{tikzpicture}
\caption{\label{fig: LorIEComp} \slshape\small Left: Center-outward multivariate (relative) Lorenz curve estimated at two different time points (1990 (blue)  and 2020 (red)).    Center: First component of that curve, which corresponds to the Education Index. Right: Second component of that curve, which corresponds to the Income Index.}
\end{figure}
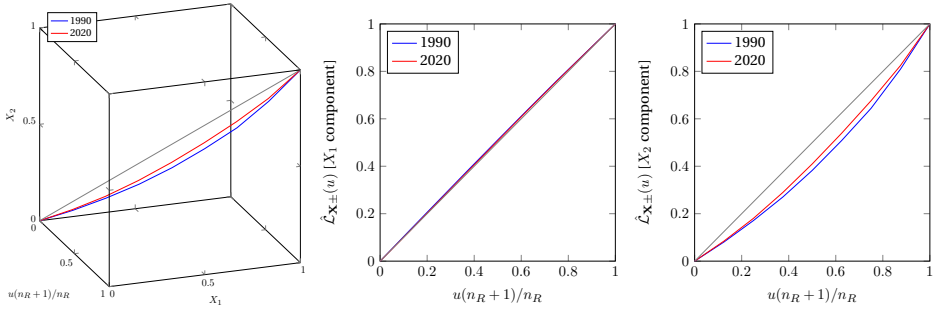

The results are shown in Figure~\ref{fig: LorIEComp}, where we provide  plots of ${\hat L}\n_{\Xb\pm}$ ($\Xb=(X_1, X_2)^\prime$\linebreak  with~$X_1=$Education and $X_2=$Income) for the years 1990 and~2019. The impact of Income clearly is highr than that of Education. These plots also indicate that the curves $\hat{\mathcal{L}}_{\Xb\pm}$,  between 1990 and 2020, have evolved towards the main diagonal of the unit cube, meaning that the empirical mean of $\Xb$ conditionally on being in a central quantile region of level~$u$ is getting closer to the overall empirical mean---less concentration, thus, in 2020 than in~1990. 

\begin{figure}[h!]
\begin{tikzpicture}[scale=0.8]
\begin{axis}[
    xlabel={$u (n_R+1)/n_R$},
    ylabel={ Kakwani function ${\hat{ \mathcal K}}^{\text{\tiny cond}}_{Y/\Xb\pm}(u)$},
    xmin=0, xmax=1,
    ymin=0.99, ymax=1.1,
    legend pos=north east
]
\addplot[
    color=blue
        ]      
   coordinates {
(0.111/0.888, 1.048125)
(0.222/0.888, 1.036622)
(0.333/0.888, 1.027369)
(0.444/0.888, 1.030250)
(0.555/0.888, 1.020707)
(0.666/0.888, 1.011598)
(0.777/0.888, 1.007211)
(0.888/0.888, 1.0000000)
    };
 \addlegendentry{ 1990}
 \addplot[
    color=darkraspberry
        ]
    coordinates {
(0.111/0.888, 1.060781)
(0.222/0.888, 1.040330)
(0.333/0.888, 1.030923)
(0.444/0.888, 1.034770)
(0.555/0.888, 1.021763)
(0.666/0.888, 1.012155)
(0.777/0.888, 1.006597)
(0.888/0.888, 1.0000000)
    };
 \addlegendentry{2000}
 \addplot[
    color=bulgarianrose,
        ]
    coordinates {
(0.111/0.888, 1.022984)
(0.222/0.888, 1.027072)
(0.333/0.888, 1.025151)
(0.444/0.888, 1.013048)
(0.555/0.888, 1.011172)
(0.666/0.888, 1.008161)
(0.777/0.888, 1.002705)
(0.888/0.888, 1.0000000)
    };
 \addlegendentry{2010}
\addplot[
    color=red
        ]
    coordinates {
(0.111/0.888, 1.018290)
(0.222/0.888, 1.010726)
(0.333/0.888, 1.011753)
(0.444/0.888, 1.014523)
(0.555/0.888, 1.008636)
(0.666/0.888, 1.006044)
(0.777/0.888, 1.001592)
(0.888/0.888, 1.0000000)
    };
 \addlegendentry{2020}
 \addplot[
    color=gray
        ]
    coordinates {
(0, 1)
(0.888/0.888, 1.0000000)
    };
\end{axis}
\end{tikzpicture}
\caption{\label{fig: Kak} \slshape\small  The empirical conditional relative Kakwani function  ${\hat{\mathcal K}}^{\text{\tiny cond}}_{Y/\Xb\pm}$  of $Y$=Health conditional on $\Xb$=(Education, Income).   The reference line (corresponding to the case under which  all values of~$u$ yield the same conditional  relative expectation value one)  is in gray. As   $n_0=0$ is zero, the   the curve is defined for $1/n_R \leq u\leq 1$. }
\end{figure}
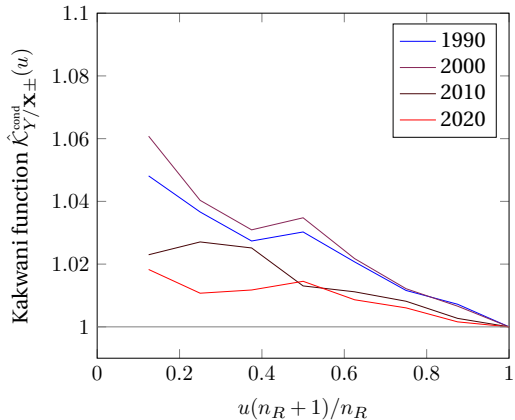

Figure~\ref{fig: Kak} yields, for the period 1990 - 2020, the graphs  of the empirical conditional relative Kakwani functions  ${\hat{\mathcal K}}^{\text{\tiny cond}}_{Y/\Xb\pm}$  of $Y=$Health conditional on $\Xb =$(Education, Income),  
which  represents, as a function of $u\in [0,1]$,  the relative ``expected health'' for a ``middle-class of dimension $u$'' characterized  in terms of national Education and Income indices.  
%
The results show, for instance,  that, conditionally on the fact that a country belonged to the~20\%  center-outward quantile region for Education and Income, its  expected health score in~1990 was much higher than for the countries in the  70\%  center-outward  quantile region. This was much less the case in 2020, though, and the evolution between 1990 and 2020, again, clearly indicates that  the discrepancies between ``middle-class" and ``outlying" (in terms of   Education and Income) countries  significantly  reduced over the past thirty-year period. 

\begin{table}[ht!]

\begin{tabular}{@{}lrrrr@{}@{}}
\hline
Year 
& \multicolumn{1}{c}{$G_{\Xb\pm}$}
& \multicolumn{1}{c}{$P_{\Xb\pm}$}
& \multicolumn{1}{c}{$G\!K_{Y/\Xb\pm}$}
& \multicolumn{1}{c}{$P\!K_{Y/\Xb\pm}$} \\
\hline
$1990$    & 0.1140 & 0.0834    & 0.0568   & 0.0134     \\
$2000$    & 0.1274 & 0.0977    &  0.0567  & 0.0155     \\
$2010$    & 0.1093 & 0.0757    &  0.0558  & 0.0084     \\
$2020$    & 0.1000 & 0.0640    &  0.0559   & 0.0065     \\
\hline
\end{tabular}
\caption{ \label{tab: coefs} \small\it  Empirical Gini and Pietra coefficients associated with the empirical Lorenz and Kakwani functions shown in Figures~\ref{fig: LorIEComp}  and~\ref{fig: Kak}.}
\end{table}
\vspace{-4mm}
Table~\ref{tab: coefs} further provides the empirical Gini and Pietra coefficients $G_{\Xb\pm}$, $P_{\Xb\pm}$, $G\!K_{Y/\Xb\pm}$, and $P\!K_{Y/\Xb\pm}$ associated with the Lorenz and Kakwani functions shown in Figures~\ref{fig: LorIEComp}  and~\ref{fig: Kak}; their variations between 1990 and 2020   are in line with the evolution of the  Lorenz and Kakwani 
functions they are summarizing.


\subsubsection{Center-outward Lorenz potential functions}
Next, we propose graphs of the Lorenz potentials defined in Equation~\eqref{eq: coDefLCPot} and estimated as in Section~\ref{sec: MultLoEst}.
As in Figure~\ref{fig: visualisation},  these potentials are for the variable $\Xb$ consisting of the Income  and  Education Indices of~144 countries in the year~2015. 
The potential increase was  steeper  in 1990 than it is in~2020, indicating that the spread of  $\Xb$ has decreased between~1990 and~2020.  
This   complements the  information provided by the Lorenz  curves of Figure~\ref{fig: LorIEComp}.

\begin{figure}[h!]
\begin{tikzpicture}[scale=0.75]
\begin{axis}[
    xlabel={$u$},
    ylabel={$\hat{\Lambda}_{\Xb\pm}$},
    xmin=0, xmax=1,
    ymin=0, ymax=0.45,
    legend pos=north west
]
\addplot[
    color=blue
        ]      
    coordinates {
    (0,0)
(0.111, 0.001738945)
(0.222, 0.011513719)
(0.333, 0.033282306)
(0.444, 0.071041953)
(0.555, 0.127301483)
(0.666, 0.204823307)
(0.777, 0.307430833)
(0.888, 0.441373439)
    };
 \addlegendentry{ 1990}
\addplot[
    color=red
        ]
    coordinates {
    (0,0)
(0.111, 0.0007754175)
(0.222, 0.0075034240)
(0.333, 0.0248323975)
(0.444, 0.0559973943)
(0.555, 0.1037853253)
(0.666, 0.1708010556)
(0.777, 0.2598894565)
(0.888, 0.3730399769)
    };
 \addlegendentry{2020}
\end{axis}
\end{tikzpicture}
\caption{\label{fig: LorPot} \slshape\small The (estimated) Lorenz potential functions associated with the joint distributions of the Education and  Income  Indices in 1990 and 2020, respectively.}
\end{figure}
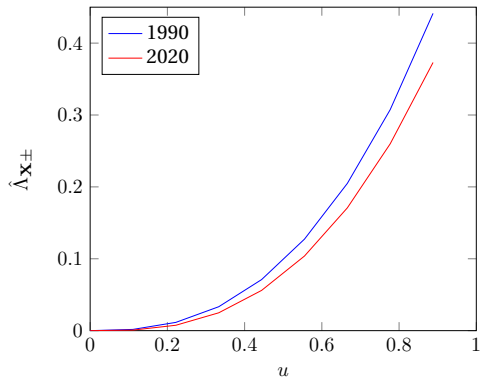


\subsection{The 2014 Spanish Health Survey}
\label{sec: App2}

\subsubsection{The dataset}
We further applied some of our concepts to the results of the European Health Survey for Spain provided by the
\citet{INE}. 
This dataset contains information about  individuals' state of health in Spain in~2014. The data is gathered from a survey that took place in a framework set by the European Statistical System. We refer to the documentation on the website of the Instituto Nacional de Estad\'istica for more information about the data collection process and the European Health survey context.

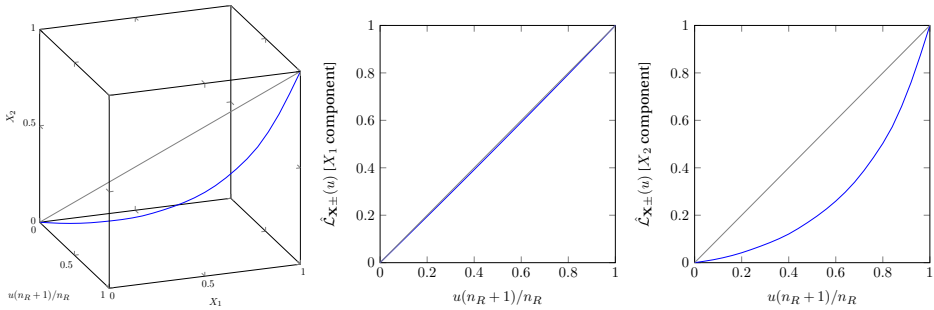
\begin{figure}[h!]
\begin{tikzpicture}[scale=0.38]
\begin{axis} [
view = {70}{20},
legend pos=north west, 
 xmin=0, xmax=1,
    ymin=0, ymax=1,
        zmin=0, zmax=1,
xtick={0,0.5,1},
ytick={0,0.5,1},
ztick={0,0.5,1},
xlabel={$u (n_R+1)/n_R$},
ylabel={$X_1$},
zlabel={$X_2$},
scale=2,
3d box= complete*,
unit vector ratio*=2 2 2
]
 
\addplot3[color= blue] table 
{
x y   z
0.00 0.00000000 0.000000000
0.04 0.03878755 0.005883336
0.08 0.07762285 0.012385091
0.12 0.11660140 0.020491392
0.16 0.15572322 0.030168812
0.20 0.19483310 0.041584489
0.24 0.23396686 0.054955708
0.28 0.27327969 0.069162627
0.32 0.31269997 0.084355674
0.36 0.35220381 0.101470834
0.40 0.39192254 0.120808959
0.44 0.43171291 0.143907739
0.48 0.47153909 0.169028915
0.52 0.51146077 0.196022063
0.56 0.55154959 0.225572455
0.60 0.59161454 0.258482367
0.64 0.63176305 0.295871636
0.68 0.67209064 0.338743105
0.72 0.71262117 0.388099616
0.76 0.75328303 0.443005181
0.80 0.79398071 0.503743941
0.84 0.83490521 0.574143406
0.88 0.87611623 0.660688618
0.92 0.91720787 0.761089754
0.96 0.95837114 0.876349657
1.00 1.00000000 1.000000000
};
 
\addplot3[color=gray]
 table 
{
x y   z
0.000 0.0000000 0.0000000
1 1.0000000 1.0000000
};
\end{axis}
\end{tikzpicture}
\begin{tikzpicture}[scale=0.55]
\begin{axis}[
    ylabel={$\hat{\mathcal{L}}_{\Xb\pm}(u)$  [$X_1$ component]},
    xlabel={$u(n_R+1)/n_R$},
    xmin=0, xmax=1,
    ymin=0, ymax=1,
    legend pos=north west,
        unit vector ratio*=1 1 
]
\addplot[
    color=blue
        ]
    coordinates {
(0.000, 0.0000000)
(0.040, 0.03880283)
(0.080, 0.07765344)
(0.120, 0.11665930)
(0.160, 0.15574877)
(0.200, 0.19483823)
(0.240 , 0.23396353)
(0.280 , 0.27329185)
(0.320 , 0.31271572)
(0.360 , 0.35222319)
(0.400 , 0.39189786)
(0.440 , 0.43170391)
(0.480 , 0.47152190)
(0.520 , 0.51144737)
(0.560 , 0.55152811)
(0.600 , 0.59163273)
(0.640 , 0.63179706)
(0.680 , 0.67216443)
(0.720 , 0.71274677)
(0.760 , 0.75343660)
(0.800 , 0.79405477)
(0.840 , 0.83493569)
(0.880 , 0.87612712)
(0.920 , 0.91718718)
(0.960 , 0.95834279)
(1 , 1)
    };
\addplot[
    color=gray
        ]
    coordinates {
(0.000, 0.0000000)
(1, 1)
    };
\end{axis}
\end{tikzpicture}
\begin{tikzpicture}[scale=0.55]
\begin{axis}[
    xlabel={$u(n_R+1)/n_R$},
    ylabel={$\hat{\mathcal{L}}_{\Xb\pm}(u)$ [$X_2$ component]},
    xmin=0, xmax=1,
    ymin=0, ymax=1,
    legend pos=north west, 
    unit vector ratio*=1 1 
]
\addplot[
    color=blue
        ]
    coordinates {
(0.000, 0.0000000)
(0.040, 0.005904059)
(0.080, 0.012428715)
(0.120, 0.020630661)
(0.160, 0.030342167)
(0.200, 0.041781281)
(0.240, 0.055031869)
(0.280, 0.069339148)
(0.320, 0.084619255)
(0.360, 0.101459242)
(0.400, 0.120748071)
(0.440, 0.143760483)
(0.480, 0.169104327)
(0.520, 0.196242872)
(0.560, 0.226182489)
(0.600, 0.259476686)
(0.640, 0.296863469)
(0.680, 0.339332439)
(0.720, 0.388393157)
(0.760, 0.443072794)
(0.800, 0.502767528)
(0.840, 0.572425361)
(0.880, 0.658990272)
(0.920, 0.759862462)
(0.960, 0.875679302)
(1, 1)
    };
\addplot[
    color=gray
        ]
    coordinates {
(0.000, 0.0000000)
(1, 1)
    };
\end{axis}
\end{tikzpicture}
\caption{\label{fig: LorAgeAbsen} \slshape\small  Left: the center-outward multivariate Lorenz curve for the bivariate distribution of~$X_1$=Age and $X_2$=Work absenteeism. Center: the first component of that Lorenz curve, which correspond to  $X_1$. Right: the second component, which  corresponds to $X_2$. In gray, the reference line which is the diagonal of the unit cube or square.}
\end{figure}

Motivated by socio-economic applications, we consider  work absenteeism (measured in days over a one-year period) and  its relation to other characteristics of the individuals. Our aim with this dataset is to show that it is possible to go beyond the bivariate examples presented in the previous section and exhibit various curve patterns.

We start by exhibiting the Lorenz center-outward curve for $\Xb = (X_1,X_2)$ where $X_1$ is   rescaled\footnote{Rescaling is performed by dividing each observation by its largest observed value, thus yielding a [0-1] range and  ensuring some balance between the components involved  in the matching problem that leads to the center-outward ranks and signs.} age and $X_2$   rescaled$^7$ work absenteeism. 
The result are shown in Figure~\ref{fig: LorAgeAbsen}. The Lorenz curve (left panel) indicates a significant concentration; that concentration, however, is due to absenteeism (as shown by the center panel) rather than age (right panel). 

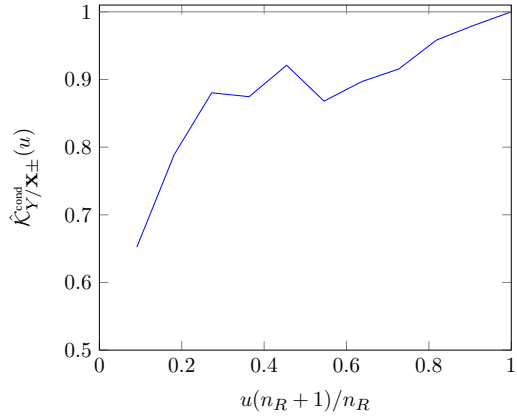
\begin{figure}[t!]
\begin{tikzpicture}[scale=0.8]
\begin{axis}[
    xlabel={$u(n_R+1)/n_R$},
    ylabel={ ${\hat{ \mathcal K}}^{\text{\tiny cond}}_{Y/\Xb\pm}(u)$ },
    xmin=0, xmax=1,
    ymin=0.5, ymax=1.01,
    legend pos=north east
]
\addplot[
    color=blue
        ]      
    coordinates {
( 0.08333333/0.91666667,  0.6522463) 
( 0.16666667/0.91666667, 0.7891015 )
( 0.25000000/0.91666667, 0.8803383 )
( 0.33333333/0.91666667, 0.8745840) 
( 0.41666667/0.91666667, 0.9210483 )
( 0.50000000/0.91666667, 0.8679978 )
( 0.58333333/0.91666667, 0.8967198 )
( 0.66666667/0.91666667, 0.9156094 )
( 0.75000000/0.91666667, 0.9581716 )
( 0.83333333/0.91666667, 0.9800333 )
( 0.91666667/0.91666667, 1.0000000)
    };
\addplot[
    color=gray
        ]
    coordinates {
( 0,  1) 
(1, 1.0000000)
    };
\end{axis}
\end{tikzpicture}
\caption{\label{fig: KakS1} \slshape\small  Relative conditional Kakwani function (multiplied by the number  of observations in the corresponding quantile region) of $Y$ =  Absenteeism   conditional on  $\Xb$ =  (Age, Weight, Alcohol Consumption). In gray, the reference line.}
\end{figure}


\begin{figure}[t!]
\begin{tikzpicture}[scale=0.8]
\begin{axis}[
    xlabel={$u(n_R+1)/n_R$},
    ylabel={ ${\hat{ \mathcal K}}^{\text{\tiny cond}}_{Y/\Xb\pm}(u)$ },
    xmin=0, xmax=1,
    ymin=0.5, ymax=1.01,
    legend pos=north east
]
\addplot[
    color=blue
        ]      
    coordinates {
( 0.1428571/0.8571429,  0.8366911) 
( 0.2857143/0.8571429, 0.8729298 )
( 0.4285714/0.8571429, 0.9078026 )
( 0.5714286/0.8571429, 0.9659595) 
( 0.7142857/0.8571429, 0.9429230 )
( 0.8571429/0.8571429, 1.0000000 )

    };
\addplot[
    color=gray
        ]
    coordinates {
( 0,  1) 
(1, 1.0000000)
    };
\end{axis}
\end{tikzpicture}
\caption{\label{fig: KakS2} \slshape\small  Relative Kakwani function  (multiplied by the number  of observations in the corresponding quantile region) of $Y$ =  Absenteeism   conditional on  $\Xb$ =  (Age, Weight, Alcohol Consumption, Sports Practice). In gray, the reference line.}
\end{figure}
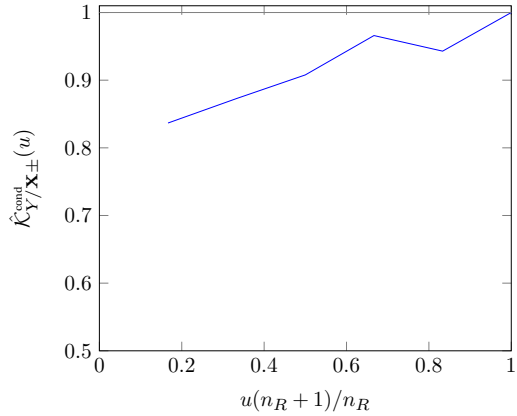

More information is provided by incorporating more variables and computing, e.g., (estimated)  conditional relative Kakwani functions (as defined  in Definition~\ref{c-oLorCo}). Figure~\ref{fig: KakS1}   provides plots of these conditional Kakwani functions for $Y$= Work Absenteeism and~$\mathbf{X} =(X_1,X_2, X_3)$ with $X_1$ = rescaled Age,  $X_2$ = rescaled Weight,  and $X_3$ = rescaled  Alcohol Consumption.  
Figure~\ref{fig: KakS2} adds  to $\Xb$ a fourth component $X_4$ = Sports Practice, the rescaled time spent on practicing sports.   The estimators are those proposed in Section~\ref{sec: MultLoEst}. 
Figure~\ref{fig: KakS1} thus describes, as a function of $u\in [0,1]$, the expected Absenteeism for individuals belonging to middle-class regions of order $u$, where the construction of these middle classes is based on  Age, Weight, and Alcohol Consumption.  For the sake of readability, the value at $u$ of the Kakwani function  has been multiplied by the number $n/(n_0+n_S \lfloor (n_R+1)  u \rfloor)\approx u^{-1}$ of observations in the quantile region of order $u$. Figure~\ref{fig: KakS2} has the same interpretation, now with $\Xb$ defining a four-dimensional ``middle class."

Both Figures~\ref{fig: KakS1} and~ \ref{fig: KakS2} indicate a concentration of Absenteeism   on ``extreme''  (in terms of Age, Weight,  Alcohol Consumption, and Sports Practice) observations, with a relatively small contribution of the corresponding central values.

\bibliographystyle{ecta-fullname} 
\bibliography{Gini}  


\end{document}